\documentclass[letterpaper,10pt]{article}
\usepackage[utf8x]{inputenc}
\usepackage[title]{appendix}
\usepackage[noTeX]{mmap}
\usepackage[left=1.3in, right=1.3in, bottom = 1.4in, top = 1.4in]{geometry}
\usepackage{hyperref} 
\hypersetup{colorlinks}
\usepackage{amsmath,amsfonts, amsthm, amssymb}
\usepackage{color}
\usepackage{enumerate}
\usepackage{hypbmsec}
\usepackage{bbm}
\usepackage{dsfont}
\usepackage{mathtools}
\usepackage{graphicx}
\usepackage{caption}
\usepackage{subcaption}
\usepackage[section]{placeins}
\usepackage{tikz}
\usetikzlibrary{calc}
\usetikzlibrary{decorations.pathreplacing}
\usetikzlibrary{arrows}
\usepackage{pgffor}
\usepackage{longtable}
\usepackage{comment}
\usepackage{multirow}
\usepackage{rotating}
\usepackage{bm}
\usepackage{xypic}
\usepackage{pgf,tikz,pgfplots}
\pgfplotsset{compat=1.15}
\pgfplotsset{ticks = none}

\usepackage{fancyhdr}

\usepackage{tocloft}

\theoremstyle{plain}
\newtheorem{theorem}{Theorem}[section]

\newtheorem{definition}[theorem]{Definition}

\theoremstyle{definition}

\newcommand{\Rn}{\mathcal{R}_n}

\newcommand{\old}[1]{}

\newcommand{\distto}{%
	\mathrel{\vbox{\offinterlineskip\ialign{%
				\hfil##\hfil\cr
				$\scriptscriptstyle\mathrm{d}$\cr
				\noalign{\kern-.05ex}
				$\to$\cr
}}}}
\newcommand{\findimto}{%
	\mathrel{\vbox{\offinterlineskip\ialign{%
				\hfil##\hfil\cr
				$\scriptscriptstyle\mathrm{f.d.}$\cr
				\noalign{\kern-.05ex}
				$\to$\cr
}}}}

\newcommand{\Probto}{%
	\mathrel{\vbox{\offinterlineskip\ialign{%
				\hfil##\hfil\cr
				$\scriptscriptstyle\Prob$\cr
				\noalign{\kern-.05ex}
				$\to$\cr
}}}}
\newcommand{\TVto}{%
	\mathrel{\vbox{\offinterlineskip\ialign{%
				\hfil##\hfil\cr
				$\scriptscriptstyle\mathrm{TV}$\cr
				\noalign{\kern-.05ex}
				$\to$\cr
}}}}

\newcommand{\N}{\mathbb{N}}
\newcommand{\Z}{\mathbb{Z}}

\newcommand{\R}{\mathbb{R}}

\newcommand{\Prob}{\mathds{P}}

\newcommand{\eqdist}{%
	\mathrel{\vbox{\offinterlineskip\ialign{%
				\hfil##\hfil\cr
				$\scriptscriptstyle\mathrm{law}$\cr
				\noalign{\kern.2ex}
				$=$\cr
}}}}

\title{Abelian sandpiles on Sierpi\'nski gasket graphs}
\date{\today}
\author{Robin Kaiser, Ecaterina Sava-Huss, Yuwen Wang \footnote{The research of all three authors is supported by the Austrian Science Fund (FWF): P 34129}}

\begin{document}
\maketitle

\begin{abstract}
The aim of the current work is to investigate structural properties of the sandpile group of a special class of self-similar graphs. More precisely, we consider Abelian sandpiles on Sierpi\'nski gasket graphs and for the choice of normal boundary conditions, we give a characterization of the identity element and a recursive description of the sandpile group. Finally, we consider Abelian sandpile Markov chains on the aforementioned graphs and we improve the existing bounds on the speed of convergence to stationarity.
\end{abstract}



\section{Introduction}

Sandpiles, as models of self-organized criticality, were introduced on lattices by  Bak, Tang and Wiesenfeld \cite{bak-tang-wiesenfeld-87}, and have been intensively studied both in physics and mathematics since then. In \cite{dhar-1990}, Dhar investigated the Abelian group structure of the addition operators in this model, generalized it to arbitrary finite graphs, and called it Abelian sandpile model. He also gave an algorithmic one-to-one correspondence between recurrent configurations of the Abelian sandpile model and rooted spanning trees of the underlying graph; this is the so-called {\it burning algorithm}. On state spaces such as Euclidean lattices, there has been impressive progress in the last decades, and many of the conjectures and numerical simulations coming from physics have been proved/disproved, but still a  variety of open questions are lacking mathematical proofs; see for instance the excellent survey \cite{sandpile-jarai} and the references therein for recent developments and open questions.
In combinatorics, these models are also refered to as {\it chip firing games}, see \cite{chip-firing-book}.

On state spaces other than lattices, sandpile models didn't receive mathematically so much attention. For instance,  Abelian sandpiles on Sierpi\'nski gasket graphs have been considered by physicists for more than 20 years ago in \cite{waves-sandpile-daerden,crit-exp-daerden-1998,KutnjakUrbanc1996SandpileMO}, where several predictions and conjectures concerning the size of avalanches, size of waves, critical exponents, and other related quantities have been made. It has been predicted in the aforementioned papers that the Abelian sandpile model exhibits peculiar behaviour and log-periodic oscillations which are, as one would expect, related to the self-similar structure and the scaling invariance of the gasket. Thus, a rigorous understanding of the sandpile group and its structural properties, of the abelian sandpile Markov chain and its speed of convergence to stationarity on fractal graphs may bring us closer in approaching mathematically the physical findings. A first step in this direction has been taken in \cite{chen-sandpile-2020}, where the authors investigate the limit shape for Abelian sandpiles on Sierpi\'nski gasket graphs. As a consequence, the Sierpi\'nski gasket graph is known to be the first nontrivial state space (other than $\mathbb{Z}$) where four aggregation models of cluster growth (internal DLA \cite{idla-gasket}, rotor-router aggregation and abelian sandpile \cite{chen-sandpile-2020}, divisible sandpile \cite{div-sandpile-gasket}) have the same limit shape.
In  \cite{chen-sandpile-2020} several properties of the sandpile group were considerd, for particular choices of sink vertices; the authors have also asked for a full characterization and other properties of the sandpile group and its identity element. This is the purpose of the underlying note: to explore the self-similar structure of the gasket in order to get additional information on the identity element and to give a recursive characterization of the sandpile group. In addition, we slightly improve the existing bounds on the mixing time from \cite{abel-sand-mix-pike-levine-jerison} for Abelian sandpile Markov chains. 

\begin{figure}
    \centering
    \begin{subfigure}[t]{0.3\linewidth}
        \centering
        \resizebox{\linewidth}{!}{
        \begin{tikzpicture}[baseline=9ex]
        \node[shape=circle,draw=black] (A) at (0,0) {};
        \node[shape=circle,draw=black] (B) at (4,0) {};
        \node[shape=circle,draw=black] (E) at (2,1.73*2) {};
        
        \node[shape=circle,draw=none] (A1) at (-1,0) {};
        \node[shape=circle,draw=none] (A2) at (-1/2,-1.73/2) {};
        \node[shape=circle,draw=none] (B1) at (5,0) {};
        \node[shape=circle,draw=none] (B2) at (4+1/2,-1.73/2) {};
        \node[shape=circle,draw=none] (C1) at (3/2,1.73*2+1.73/2) {};
        \node[shape=circle,draw=none] (C2) at (5/2,1.73*2+1.73/2) {};
        
        \path [-] (A) edge node[left] {} (B);
        \path [-] (A) edge node[left] {} (E);
        \path [-] (B) edge node[left] {} (E);
        
        \path [-] (A) edge node[left] {} (A1);
        \path [-] (A) edge node[left] {} (A2);
        \path [-] (B) edge node[left] {} (B1);
        \path [-] (B) edge node[left] {} (B2);
        \path [-] (E) edge node[left] {} (C1);
        \path [-] (E) edge node[left] {} (C2);
        \end{tikzpicture}
        }
    \end{subfigure}
    \begin{subfigure}[t]{0.3\linewidth}
        \centering
        \resizebox{\linewidth}{!}{
        \begin{tikzpicture}[baseline=9ex]
        \node[shape=circle,draw=black] (A) at (0,0) {};
        \node[shape=circle,draw=black] (B) at (2,0) {};
        \node[shape=circle,draw=black] (C) at (4,0) {};
        \node[shape=circle,draw=black] (D) at (1,1.73) {};
        \node[shape=circle,draw=black] (E) at (3,1.73) {};
        \node[shape=circle,draw=black] (F) at (2,1.73*2) {} ;
        
        \node[shape=circle,draw=none] (A1) at (-1,0) {};
        \node[shape=circle,draw=none] (A2) at (-1/2,-1.73/2) {};
        \node[shape=circle,draw=none] (B1) at (5,0) {};
        \node[shape=circle,draw=none] (B2) at (4+1/2,-1.73/2) {};
        \node[shape=circle,draw=none] (C1) at (3/2,1.73*2+1.73/2) {};
        \node[shape=circle,draw=none] (C2) at (5/2,1.73*2+1.73/2) {};
        
        \path [-] (A) edge node[left] {} (B);
        \path [-] (A) edge node[left] {} (D);
        \path [-] (B) edge node[left] {} (D);
        \path [-] (B) edge node[left] {} (C);
        \path [-] (E) edge node[left] {} (B);
        \path [-] (E) edge node[left] {} (C);
        \path [-] (F) edge node[left] {} (E);
        \path [-] (F) edge node[left] {} (D);
        \path [-] (E) edge node[left] {} (D);
        
        \path [-] (A) edge node[left] {} (A1);
        \path [-] (A) edge node[left] {} (A2);
        \path [-] (C) edge node[left] {} (B1);
        \path [-] (C) edge node[left] {} (B2);
        \path [-] (F) edge node[left] {} (C1);
        \path [-] (F) edge node[left] {} (C2);
        \end{tikzpicture}}
    \end{subfigure}
    \begin{subfigure}[t]{0.3\linewidth}
        \centering
        \resizebox{\linewidth}{!}{
        \begin{tikzpicture}[baseline=9ex]
        \node[shape=circle,draw=black] (A) at (0,0) {};
        \node[shape=circle,draw=black] (B) at (2,0) {};
        \node[shape=circle,draw=black] (C) at (4,0) {};
        \node[shape=circle,draw=black] (D) at (1,1.73) {};
        \node[shape=circle,draw=black] (E) at (3,1.73) {};
        \node[shape=circle,draw=black] (F) at (2,1.73*2) {} ;
        \node[shape=circle,draw=black] (G) at (1,0) {};
        \node[shape=circle,draw=black] (H) at (3,0) {};
        \node[shape=circle,draw=black] (I) at (1/2,1.73/2) {};
        \node[shape=circle,draw=black] (J) at (3/2,1.73/2) {};
        \node[shape=circle,draw=black] (K) at (5/2,1.73/2) {};
        \node[shape=circle,draw=black] (L) at (7/2,1.73/2) {};
        \node[shape=circle,draw=black] (M) at (3/2,1.73/2+1.73) {};
        \node[shape=circle,draw=black] (N) at (5/2,1.73/2+1.73) {};
        \node[shape=circle,draw=black] (O) at (2,1.73) {};
        
        \node[shape=circle,draw=none] (A1) at (-1,0) {};
        \node[shape=circle,draw=none] (A2) at (-1/2,-1.73/2) {};
        \node[shape=circle,draw=none] (B1) at (5,0) {};
        \node[shape=circle,draw=none] (B2) at (4+1/2,-1.73/2) {};
        \node[shape=circle,draw=none] (C1) at (3/2,1.73*2+1.73/2) {};
        \node[shape=circle,draw=none] (C2) at (5/2,1.73*2+1.73/2) {};
        
        \path [-] (A) edge node[left] {} (G);
        \path [-] (B) edge node[left] {} (G);
        \path [-] (A) edge node[left] {} (I);
        \path [-] (D) edge node[left] {} (I);
        \path [-] (B) edge node[left] {} (J);
        \path [-] (D) edge node[left] {} (J);
        \path [-] (I) edge node[left] {} (J);
        \path [-] (I) edge node[left] {} (G);
        \path [-] (J) edge node[left] {} (G);
        
        \path [-] (B) edge node[left] {} (H);
        \path [-] (C) edge node[left] {} (H);
        \path [-] (E) edge node[left] {} (K);
        \path [-] (E) edge node[left] {} (L);
        \path [-] (B) edge node[left] {} (K);
        \path [-] (C) edge node[left] {} (L);
        \path [-] (H) edge node[left] {} (K);
        \path [-] (H) edge node[left] {} (L);
        \path [-] (L) edge node[left] {} (K);
        
        \path [-] (F) edge node[left] {} (M);
        \path [-] (F) edge node[left] {} (N);
        \path [-] (E) edge node[left] {} (N);
        \path [-] (E) edge node[left] {} (O);
        \path [-] (D) edge node[left] {} (O);
        \path [-] (D) edge node[left] {} (M);
        \path [-] (M) edge node[left] {} (N);
        \path [-] (O) edge node[left] {} (N);
        \path [-] (O) edge node[left] {} (M);
        
        \path [-] (A) edge node[left] {} (A1);
        \path [-] (A) edge node[left] {} (A2);
        \path [-] (C) edge node[left] {} (B1);
        \path [-] (C) edge node[left] {} (B2);
        \path [-] (F) edge node[left] {} (C1);
        \path [-] (F) edge node[left] {} (C2);
        \end{tikzpicture}}
    \end{subfigure}
    \caption{The first three iterations $G_0,G_1,G_2$ of the Sierpi\'nski gasket with normal boundary conditions.}
    \label{fig:sierp_constr}
\end{figure}
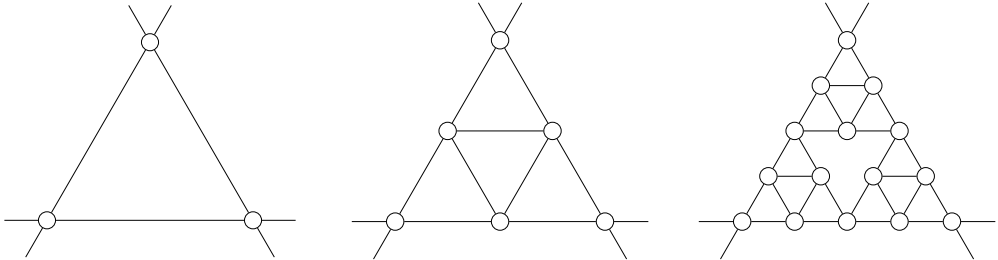

The paper is structured as following. In Section \ref{sec:prelim} we introduce the state spaces and the Abelian sandpile models. Then we fix for the rest of the paper the state spaces $G_n$, for $n\in\N$ as the $n$-th level prefractal  Sierpi\'nski gasket graph, in which each of the three corner vertices is joined by two edges with an additional vertex $s$ called the sink, and we refer to this construction as {\it $G_n$ with normal boundary conditions}; see Figure \ref{fig:sierp_constr}. Section \ref{sec:results} is dedicated to the main results of this paper, which we briefly summarize here:
\begin{itemize}
\item in Theorem \ref{thm:id} we describe recursively the identity element $\mathsf{id}_n$ of $G_n$ with normal boundary conditions, by matching three recurrent configurations on $G_{n-1}$ in a rotated fashion, and setting a fixed value at the junction points.
\item in Theorem \ref{thm:sand-group} we give a characterization of the sandpile group of $G_n$ as a direct sum of three normal subgroups of the sandpile group of $G_{n-1}$, corresponding to the three subcopies of $G_{n-1}$ that build $G_n$.
\item in Theorem \ref{thm:mix-time} we consider Abelian sandpile Markov chains which are random walks on the finite Abelian group of critical sandpile configurations of $G_n$. For such random walks, we give bounds for the speed of convergence to stationarity, that is, we show that the order of the mixing time is $|V_n|\log|V_n|$, where $V_n$ represents the set of vertices of $G_n$.
\end{itemize}

\section{Preliminaries}\label{sec:prelim}

This section is devoted to introducing the notation and preliminaries. We start by defining Abelian sandpile models and the  underlying state spaces for them: Sierpi\'nski gasket graphs.

\subsection{Abelian sandpiles on finite graphs}

Let $G=(V\cup\{s\},E)$ be a finite, connected, simple graph, with vertex set $V\cup\{s\}$, edge set $E$, with $|V|=N$, $N\in\mathbb{N}$ and with a designated vertex $s$ called the {\it sink}. For a
vertex $v\in V$, we denote by $d_v$ the degree of $v$, that is, the number of neighbours $w\in V$ of $v$ in the graph $G$.
Sometimes, it will be useful to fix an ordering of the vertices $V=\{v_1,\ldots,v_N\}$. For vertices $v,w\in V$, we write both $x\sim y$ and $(x,y)\in E$ to denote that $x$ and $y$ are connected by an edge, so $\sim$ denotes the neighbourhood relation in $G$. 

A {\it sandpile configuration} or simply a {\it sandpile} is a function $\eta:V\to\mathbb{N}$ from the nonsink vertices to the nonnegative integers. So $\eta\in \mathbb{Z}^V$ can be seen as a $N$-dimensional integer valued vector indexed over the non-sink vertices, and $\eta(v)$ represents the number of chips (or sand particles) sitting in $v$. The sandpile $\eta$ is called {\it stable} if $\eta(v)<d_v$ for every $v\in V$. Otherwise, it is called {\it unstable} and it may be stabilized by {\it toppling or firing} vertices.  A vertex topples by sending one chip to each of the neighbours, and this results in a new configuration $\eta'$ where the entry corresponding to $v$ has decreased by $d_v$, and $d_v$ entries corresponding to neighbors of $v$ have increased by one. Toppling $v$ may cause other vertices to become unstable and this may further lead to other topplings; any chip that falls into the sink is gone forever so the sink may be regarded as the vertex collecting the excess mass. The assumption that $G$ is connected, together with the existence of a sink that collects excess mass, ensures that starting with any initial sandpile configuration, we can reach a final stable configuration in finitely many steps by successive firings at vertices that are unstable. One can show that starting with an unstable sandpile configuration $\eta$ on $G$, the corresponding final configuration that we denote $\eta^{\circ}$ does not depend on the order in which the topplings are performed, and hence the terminology of Abelian sandpile \cite{dhar-1990}. Topplings are encoded in the {\it graph Laplacian $\overline{\Delta}$ of $G$} defined as the $(N+1)\times(N+1)$ matrix indexed over the vertices of $G$, with entries given by: 
\begin{align*}
    \overline{\Delta}(v,w)=
    \begin{cases}
    d_v&,v=w\\
    -1&,v\sim w\\
    0&,\text{else}
    \end{cases}.
\end{align*}
Denote by $\Delta$ the {\it reduced graph Laplacian} of $G$, which is a $N\times N$ matrix obtained from $\overline{\Delta}$ by deleting the row and column corresponding to the sink $s$. Notice that in $\overline{\Delta}$ all rows sums are zero, so $\overline{\Delta}$ is not invertible, but $\Delta$ is invertible.
In terms of $\Delta$, toppling the configuration $\eta$ at vertex $v$ results in a configuration $\eta'=\eta-\Delta \delta_v$,
where $\delta_v$ is the configuration in $\mathbb{Z}^V$ with $1$ at position corresponding to $v$, and all other entries are $0$. We define the sum of two sandpile configurations by $(\eta+\sigma)(v)=\eta(v)+\sigma(v)$ and we denote the set of stable configurations of $G$ by $\mathsf{Stable}(G)$. Moreover, we define the binary operation  $''\oplus''$ of addition of two configurations followed by stabilization by:
$$\eta\oplus \sigma=(\eta+\sigma)^{\circ}.$$
Then $\left(\mathsf{Stable}(G),\oplus\right)$ is a commutative monoid with identity being the all zero configuration.
\vspace{-0.3cm}
\paragraph{Sandpile Markov chains.} Everything so far was deterministic, but one can add randomness to the system by performing a random walk on the set $\mathsf{Stable}(G)$ of stable sandpile configurations of $G$. For this, we choose $\mu$ to be any probability measure on $V$. Given an initial state $\eta_0\in\mathsf{Stable}(G)$, pick a vertex $v\in V$ according to the  distribution $\mu$, add a particle at $v$ and stabilize: that is $\eta_0$ transitions to $\eta_1$
$$\eta_0\mapsto\eta_1=\eta_0\oplus \delta_v=(\eta_0+\delta_v)^{\circ},\quad \text{with probability }\mu(v).$$ Proceeding in this way, we obtain a Markov chain $(\eta_t)_{t\in\mathbb{N}}$ with state space $\mathsf{Stable}(G)$,
defined as: for any $t\geq 1$
$$\eta_t=\eta_{t-1}\oplus \delta_{X_{t}},$$
where $(X_t)_{t\geq 1}$ is a sequence of iid random variables distributed according to $\mu$. The choice of  $\mu$ during this work is the uniform distribution over $V$, i.e. at each step we choose one vertex uniformly at random, add a chip there and stabilize this configuration. The Markov chain $(\eta_t)_{t\in\N}$ over the stable configurations of $G$ is called {\it the Abelian sandpile Markov chain}.
\vspace{-0.3cm}
\paragraph{Recurrent configurations and the sandpile group.}
We recall that for a Markov chain a state $\eta$ is called recurrent if, starting from $\eta$, the Markov chain returns to $\eta$ with probability one. Otherwise, the state is called transient. It is known, see \cite{sandpile-jarai} for a survey on Abelian sandpiles, \cite{dhar-1990}, and \cite{levine-survey}, that the {\it Abelian sandpile Markov chain} has exactly one recurrent communicating class that we denote $\mathcal{R}(G)$ and a configuration $\eta\in\mathcal{R}(G)$ if and only if it can be reached from the {\it saturated or maximal configuration} $\eta^{\max}$ defined as $\eta^{\max}(v)=d_v-1$, for all $v\in V$. So one can represent the set of recurrent configurations over $G$ by
$$\mathcal{R}(G)=\left\{\eta\in\mathsf{Stable(G)}:\ \eta=\eta^{\max}\oplus\sigma,\text{ for some }\sigma\in\mathsf{Stable(G)}\right\}.$$ The set $(\mathcal{R}(G),\oplus)$ is called {\it the sandpile group} and its elements are also called critical configurations. One can easily check that $(\mathcal{R}(G),\oplus)$ is a nonempty abelian group and it is the minimal ideal of $\mathsf{Stable}(G)$. We denote by $\mathsf{id}$ the identity in $\mathcal{R}(G)$. Even computing the identity element of this group for a specific graph may be very involved, and another characterization (isomorphism) of $\mathcal{R}(G)$ turns out to be useful in many cases. 

We recall here two other possible characterizations of the sandpile group. Let $\Delta\mathbb{Z}^V$ be the integer row span of the reduced Laplacian $\Delta$ of $G$, which is a subgroup of the Abelian group $\mathbb{Z}^V$. We define an equivalence relation $"\sim"$ on $\mathbb{Z}^V$ as following: two configurations $\eta,\eta'\in \mathbb{Z}^V$ are $"\sim"$ -equivalent if $\eta-\eta'\in \Delta\mathbb{Z}^V$, i.e. one configuration can be obtained from another one by successive topplings. The equivalence classes under $"\sim"$ form an Abelian group, the factor group
$$\Gamma:=\mathbb{Z}^V / \Delta\mathbb{Z}^V.$$
It is known \cite{dhar-1990} that every equivalence class in $\Gamma$ contains precisely one recurrent sandpile configuration, that is, we have $\mathcal{R}(G)\cong \Gamma$ and  an isomorphism can be given by
$\eta\mapsto (\eta(v_1),\ldots\eta(v_{|V|}))+\Delta\mathbb{Z}^V$. In particular $|\mathcal{R}(G)|=|\Gamma|=det(\Delta)$,
where $det(\Delta)$ is the number of spanning trees of $G$ by the matrix-tree theorem.
\vspace{-0.3cm}
\paragraph{Burning algorithm} called also {\it Dhar's multiplication by identity test} \cite{dhar-1990} checks whether a configuration is recurrent or not. For any vertex $v\in V$, denote by $\beta(v)$ the number of edges in $G$ that connect $v$ to the sink $s$. The burning algorithm states that a sandpile configuration $\eta$ is recurrent if and only if adding $\beta(v)$ chips at each vertex $v$ causes every vertex to topple exactly once and after stabilization the same configuration $\eta$ is returned:
$$\eta\oplus \sum_{v\in V}\beta(v)\delta_v=\eta.$$
We write $[\eta]$ for the equivalence class containing $\eta$.
The burning algorithm has been applied to junction points on the gasket in \cite{chen-sandpile-2020}, in order to produce self-similar sandpile tiles. 

\vspace{-0.3cm}
\paragraph{Stationary distribution.}
Consider now the sandpile Markov chain $(\eta_n)_{n\in\N}$ whose recurrent states are $\mathcal{R}(G)\cong \Gamma=\mathbb{Z}^V / \Delta\mathbb{Z}^V$. Since the transient states will be visited only finitely many times and the Markov chain will end up in $\mathcal{R}(G)$, it makes sense to start the chain directly in a recurrent state. Then the process of adding to a recurrent configuration $\eta$ a chip at a vertex chosen uniformly at random 
can be represented as a random walk on the group $\mathcal{R}(G)$ driven by the uniform distribution on the set $S=\{\delta_v\oplus\mathsf{id} \}_{v\in V\cup\{s\}}$, where as above $\delta_v$ is the sandpile configuration on $G$ with one chip at $v$ and zero elsewhere. So $\mathcal{R}(G)$ is an abelian group generated by $S$, and in conclusion the random walk driven by the uniform distribution over $S$ is irreducible, has an unique stationary distribution $\pi=\mathsf{Unif}(\mathcal{R}(G))$ which is the uniform distribution over $\mathcal{R}(G)$; see \cite{saloff-coste-survey} for more details on this matter. In view of the isomorphism $\mathcal{R}(G)\cong \Gamma=\mathbb{Z}^V / \Delta\mathbb{Z}^V$ one can actually view the random walk $(\mathcal{R}(G),\mathsf{Unif}(S))$ as the random walk on $\Gamma$ driven by the uniform distribution on $\{\mathsf{e_1},\ldots\mathsf{e}_N,\mathsf{0}\}$ where $\mathsf{e}_i$ is the standard basis vector in $\mathbb{R}^N$, for $i=1,\ldots,N$ and $\mathsf{0}$ is the zero configuration. 

\vspace{-0.3cm}
\paragraph{Multiplicative harmonic functions.} Because one can regard the sandpile Markov chain as a random walk on the group $\Gamma=\Z^V\slash \Delta\Z^V$, one can compute the eigenvalues and the eigenfunctions of the transition matrix $\mathsf{P}$ in terms of the characters of $\Gamma$. See \cite{abel-sand-mix-pike-levine-jerison} for an exposition in this direction. The characters of $\Gamma$ are indexed by the {\it multiplicative harmonic functions of $G$}, which are functions $h:V\cup\{s\}\to \mathbb{T}$, where $\mathbb{T}$ represents the unit circle, that satisfy $h(s)=1$ and the {\it geometric mean value property}:
for all $v\in V$
\begin{equation}\label{eq:def-mult-harm}
h(v)^{d_v}=\prod_{w\sim v}h(w).
\end{equation}
If $\mathcal{H}$ denotes the set of multiplicative harmonic functions of $G$, then it is shown in \cite{abel-sand-mix-pike-levine-jerison} that $\mathcal{H}\cong \Gamma$, an isomorhism being given by $h\mapsto \chi_h$ and $\chi_h:\Gamma\to\mathbb{T}$ with $\chi_h(\mathsf{x}+\Delta\Z^V)=\prod_{v\in V}h(v)^{\mathsf{x}_v}$; here $\mathsf{x}_v$ represents the entry of the vector $\mathsf{x}$ corresponding to the vertex $v$. Thus another way to understand the behaviour of the sandpile Markov chain is through its multiplicative harmonic functions. In particular, by  \cite[Theorem 2.6]{abel-sand-mix-pike-levine-jerison} the characters of $\Gamma$ are the functions\\ $\{\chi_h:\mathcal{R}(G)\to\mathbb{T}\}_{h\in \mathcal{H}}$ defined as
\begin{equation}\label{eq:char-gamma}
\chi_h(\eta)=\prod_{v\in V}h(v)^{\eta(v)}
\end{equation}
and they represent an orthonormal basis of eigenfunctions for the transition matrix $\mathsf{P}$ of the sandpile chain. The eigenvalue associated with $\chi_h$ is 
$$\lambda_h=\frac{1}{|V|+1}\left(\sum_{v\in V}h(v)+1\right).$$ We refer once again the reader to \cite{abel-sand-mix-pike-levine-jerison}  for a beautiful exposition on multiplicative harmonic functions and their relation to the sandpile Markov chain, in particular, on how to use the properties of such functions in order to bound the speed of convergence to stationarity of the sandpile Markov chain. 

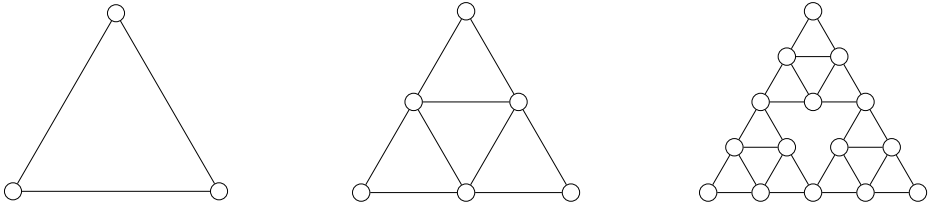
\begin{figure}
    \centering
    \begin{subfigure}[t]{0.3\linewidth}
        \centering
        \resizebox{\linewidth}{!}{
        \begin{tikzpicture}[baseline=9ex]
        \node[shape=circle,draw=black] (A) at (0,0) {};
        \node[shape=circle,draw=black] (B) at (4,0) {};
        \node[shape=circle,draw=black] (E) at (2,1.73*2) {};
        
        \node[shape=circle,draw=none] (A1) at (-1,0) {};
        \node[shape=circle,draw=none] (A2) at (-1/2,-1.73/2) {};
        \node[shape=circle,draw=none] (B1) at (5,0) {};
        \node[shape=circle,draw=none] (B2) at (4+1/2,-1.73/2) {};
        \node[shape=circle,draw=none] (C1) at (3/2,1.73*2+1.73/2) {};
        \node[shape=circle,draw=none] (C2) at (5/2,1.73*2+1.73/2) {};
        
        \path [-] (A) edge node[left] {} (B);
        \path [-] (A) edge node[left] {} (E);
        \path [-] (B) edge node[left] {} (E);
        \end{tikzpicture}
        }
    \end{subfigure}
    \begin{subfigure}[t]{0.3\linewidth}
        \centering
        \resizebox{\linewidth}{!}{
        \begin{tikzpicture}[baseline=9ex]
        \node[shape=circle,draw=black] (A) at (0,0) {};
        \node[shape=circle,draw=black] (B) at (2,0) {};
        \node[shape=circle,draw=black] (C) at (4,0) {};
        \node[shape=circle,draw=black] (D) at (1,1.73) {};
        \node[shape=circle,draw=black] (E) at (3,1.73) {};
        \node[shape=circle,draw=black] (F) at (2,1.73*2) {} ;
        
        \node[shape=circle,draw=none] (A1) at (-1,0) {};
        \node[shape=circle,draw=none] (A2) at (-1/2,-1.73/2) {};
        \node[shape=circle,draw=none] (B1) at (5,0) {};
        \node[shape=circle,draw=none] (B2) at (4+1/2,-1.73/2) {};
        \node[shape=circle,draw=none] (C1) at (3/2,1.73*2+1.73/2) {};
        \node[shape=circle,draw=none] (C2) at (5/2,1.73*2+1.73/2) {};
        
        \path [-] (A) edge node[left] {} (B);
        \path [-] (A) edge node[left] {} (D);
        \path [-] (B) edge node[left] {} (D);
        \path [-] (B) edge node[left] {} (C);
        \path [-] (E) edge node[left] {} (B);
        \path [-] (E) edge node[left] {} (C);
        \path [-] (F) edge node[left] {} (E);
        \path [-] (F) edge node[left] {} (D);
        \path [-] (E) edge node[left] {} (D);
        \end{tikzpicture}}
    \end{subfigure}
    \begin{subfigure}[t]{0.3\linewidth}
        \centering
        \resizebox{\linewidth}{!}{
        \begin{tikzpicture}[baseline=9ex]
        \node[shape=circle,draw=black] (A) at (0,0) {};
        \node[shape=circle,draw=black] (B) at (2,0) {};
        \node[shape=circle,draw=black] (C) at (4,0) {};
        \node[shape=circle,draw=black] (D) at (1,1.73) {};
        \node[shape=circle,draw=black] (E) at (3,1.73) {};
        \node[shape=circle,draw=black] (F) at (2,1.73*2) {} ;
        \node[shape=circle,draw=black] (G) at (1,0) {};
        \node[shape=circle,draw=black] (H) at (3,0) {};
        \node[shape=circle,draw=black] (I) at (1/2,1.73/2) {};
        \node[shape=circle,draw=black] (J) at (3/2,1.73/2) {};
        \node[shape=circle,draw=black] (K) at (5/2,1.73/2) {};
        \node[shape=circle,draw=black] (L) at (7/2,1.73/2) {};
        \node[shape=circle,draw=black] (M) at (3/2,1.73/2+1.73) {};
        \node[shape=circle,draw=black] (N) at (5/2,1.73/2+1.73) {};
        \node[shape=circle,draw=black] (O) at (2,1.73) {};
        
        \node[shape=circle,draw=none] (A1) at (-1,0) {};
        \node[shape=circle,draw=none] (A2) at (-1/2,-1.73/2) {};
        \node[shape=circle,draw=none] (B1) at (5,0) {};
        \node[shape=circle,draw=none] (B2) at (4+1/2,-1.73/2) {};
        \node[shape=circle,draw=none] (C1) at (3/2,1.73*2+1.73/2) {};
        \node[shape=circle,draw=none] (C2) at (5/2,1.73*2+1.73/2) {};
        
        \path [-] (A) edge node[left] {} (G);
        \path [-] (B) edge node[left] {} (G);
        \path [-] (A) edge node[left] {} (I);
        \path [-] (D) edge node[left] {} (I);
        \path [-] (B) edge node[left] {} (J);
        \path [-] (D) edge node[left] {} (J);
        \path [-] (I) edge node[left] {} (J);
        \path [-] (I) edge node[left] {} (G);
        \path [-] (J) edge node[left] {} (G);
        
        \path [-] (B) edge node[left] {} (H);
        \path [-] (C) edge node[left] {} (H);
        \path [-] (E) edge node[left] {} (K);
        \path [-] (E) edge node[left] {} (L);
        \path [-] (B) edge node[left] {} (K);
        \path [-] (C) edge node[left] {} (L);
        \path [-] (H) edge node[left] {} (K);
        \path [-] (H) edge node[left] {} (L);
        \path [-] (L) edge node[left] {} (K);
        
        \path [-] (F) edge node[left] {} (M);
        \path [-] (F) edge node[left] {} (N);
        \path [-] (E) edge node[left] {} (N);
        \path [-] (E) edge node[left] {} (O);
        \path [-] (D) edge node[left] {} (O);
        \path [-] (D) edge node[left] {} (M);
        \path [-] (M) edge node[left] {} (N);
        \path [-] (O) edge node[left] {} (N);
        \path [-] (O) edge node[left] {} (M);
        \end{tikzpicture}}
    \end{subfigure}
    \caption{The graphs $\mathcal{G}_0$, $\mathcal{G}_1$ and $\mathcal{G}_2$.}
    \label{fig:sierp_constr_1}
\end{figure}
\vspace{-0.3cm}
\paragraph{Mixing time.}Once we know that the Abelian sandpile Markov chains on $\mathcal{R}(G)$ converges to the uniform distribution, it is natural to ask  about the speed of convergence, i.e. its mixing time; see \cite{abel-sand-mix-pike-levine-jerison} for bounds on mixing time for sandpile Markov chains on finite graphs. 

For two measures $\mu$ and $\nu$ on $\mathcal{R}(G)$, the {\it $\mathcal{L}^2$ distance} between them is defined as
$$\|\mu-\nu\|_2=\Big(\sum_{\eta\in\mathcal{R}(G)}|\mu(\eta)-\nu(\eta)|^2\Big)^{1/2},$$
the {\it total variation distance} as
$$\|\mu-\nu\|_{\mathsf{TV}}=\frac{1}{2}\sum_{\eta\in\mathcal{R}(G)}|\mu(\eta)-\nu(\eta)|=\sup_{R\subset \mathcal{R}(G)}|\mu(R)-\nu(R)|$$
and Cauchy-Schwarz inequality gives $\|\mu-\nu\|_{\mathsf{TV}}\leq \frac{1}{2}\|\mu-\nu\|_2$.
For a random walk on a group, the distance to stationarity of the distribution at time $t$ is independent of the initial state, so for the sandpile Markov chain $(\eta_t)_{t\in \N}$ with transition matrix $\mathsf{P}$ over $\mathcal{R}(G)$, we can assume it starts from a deterministic state, for instance from $\mathsf{id}$. The {\it mixing time} of the sandpile chain is defined as: for any $\varepsilon>0$
$$t_{\mathsf{mix}}(\varepsilon)=\min\{t\in \N:\ \|\mathsf{P}^t\delta_{\mathsf{id}}-\pi\|_{\mathsf{TV}}\leq \varepsilon\}$$
where $\pi$ is the stationary distribution over $\mathcal{R}(G)$, which is the uniform distribution. That is, the mixing time  
is the first time when the total variation distance between the distribution of the sandpile chain at time $t$ and the stationary distribution drops below $\varepsilon$. It is standard to take $\varepsilon=\frac{1}{4}$, and in this case we write only $t_{\mathsf{mix}}$ instead of $t_{\mathsf{mix}}(1/4)$.
There are several methods to obtain bounds on the mixing times for Markov chains; see 
\cite{peres-mixing-book} for a variety of approaches and methods. In particular, understanding the spectral properties of the transition matrix $\mathsf{P}$, gives us information on the speed of convergence to stationarity. For the matrix $\mathsf{P}$ with largest eigenvalue one, we denote by $\lambda^{\star}$ the size of the second largest eigenvalue, and by $\gamma^{\star}=1-\lambda^{\star}$ the spectral gap of $\mathsf{P}$. In terms of multiplicative harmonic functions, $\lambda^{\star}=\max\{h\in\mathcal{H}\backslash\{1\}:\ |\lambda_h|\}$. The relaxation time is denoted by $t_{\mathsf{rel}}=1/\gamma^{\star}$.

\subsection{The Sierpi\'nski gasket graph}

\begin{figure}
    \centering
    \begin{subfigure}[t]{0.2\linewidth}
        \centering
        \begin{tikzpicture}[baseline=9ex]
        \node[shape=circle,draw=black] (A) at (0,0) {};
        \node[shape=circle,draw=black] (C) at (1,0) {};
        \node[shape=circle,draw=black] (F) at (0.5,1.73/2) {} ;
        
        \path [-] (A) edge node[left] {} (C);
        \path [-] (A) edge node[left] {} (F);
        \path [-] (F) edge node[left] {} (C);
        \end{tikzpicture}
        \end{subfigure}
        \begin{subfigure}[t]{0.3\linewidth}
        \centering
       
         \begin{tikzpicture}[baseline=9ex]
        \node[shape=circle,draw=black] (A) at (0,0) {};
        \node[shape=circle,draw=black] (B) at (1,0) {};
        \node[shape=circle,draw=black] (C) at (2,0) {};
        \node[shape=circle,draw=black] (D) at (0.5,1.73/2) {};
        \node[shape=circle,draw=black] (E) at (1.5,1.73/2) {};
        \node[shape=circle,draw=black] (F) at (1,1.73) {} ;
        
        \path [-] (A) edge node[left] {} (B);
        \path [-] (A) edge node[left] {} (D);
        \path [-] (B) edge node[left] {} (D);
        \path [-] (B) edge node[left] {} (C);
        \path [-] (E) edge node[left] {} (B);
        \path [-] (E) edge node[left] {} (C);
        \path [-] (F) edge node[left] {} (E);
        \path [-] (F) edge node[left] {} (D);
        \path [-] (E) edge node[left] {} (D);
        \end{tikzpicture}
        \end{subfigure}
        \begin{subfigure}[t]{0.4\linewidth}
        \centering
        \begin{tikzpicture}[baseline=9ex]
        \node[shape=circle,draw=black] (A) at (0,0) {};
        \node[shape=circle,draw=black] (B) at (2,0) {};
        \node[shape=circle,draw=black] (C) at (4,0) {};
        \node[shape=circle,draw=black] (D) at (1,1.73) {};
        \node[shape=circle,draw=black] (E) at (3,1.73) {};
        \node[shape=circle,draw=black] (F) at (2,1.73*2) {} ;
        \node[shape=circle,draw=black] (G) at (1,0) {};
        \node[shape=circle,draw=black] (H) at (3,0) {};
        \node[shape=circle,draw=black] (I) at (1/2,1.73/2) {};
        \node[shape=circle,draw=black] (J) at (3/2,1.73/2) {};
        \node[shape=circle,draw=black] (K) at (5/2,1.73/2) {};
        \node[shape=circle,draw=black] (L) at (7/2,1.73/2) {};
        \node[shape=circle,draw=black] (M) at (3/2,1.73/2+1.73) {};
        \node[shape=circle,draw=black] (N) at (5/2,1.73/2+1.73) {};
        \node[shape=circle,draw=black] (O) at (2,1.73) {};
        
        \node[shape=circle,draw=none] (A1) at (-1,0) {};
        \node[shape=circle,draw=none] (A2) at (-1/2,-1.73/2) {};
        \node[shape=circle,draw=none] (B1) at (5,0) {};
        \node[shape=circle,draw=none] (B2) at (4+1/2,-1.73/2) {};
        \node[shape=circle,draw=none] (C1) at (3/2,1.73*2+1.73/2) {};
        \node[shape=circle,draw=none] (C2) at (5/2,1.73*2+1.73/2) {};
        
        \path [-] (A) edge node[left] {} (G);
        \path [-] (B) edge node[left] {} (G);
        \path [-] (A) edge node[left] {} (I);
        \path [-] (D) edge node[left] {} (I);
        \path [-] (B) edge node[left] {} (J);
        \path [-] (D) edge node[left] {} (J);
        \path [-] (I) edge node[left] {} (J);
        \path [-] (I) edge node[left] {} (G);
        \path [-] (J) edge node[left] {} (G);
        
        \path [-] (B) edge node[left] {} (H);
        \path [-] (C) edge node[left] {} (H);
        \path [-] (E) edge node[left] {} (K);
        \path [-] (E) edge node[left] {} (L);
        \path [-] (B) edge node[left] {} (K);
        \path [-] (C) edge node[left] {} (L);
        \path [-] (H) edge node[left] {} (K);
        \path [-] (H) edge node[left] {} (L);
        \path [-] (L) edge node[left] {} (K);
        
        \path [-] (F) edge node[left] {} (M);
        \path [-] (F) edge node[left] {} (N);
        \path [-] (E) edge node[left] {} (N);
        \path [-] (E) edge node[left] {} (O);
        \path [-] (D) edge node[left] {} (O);
        \path [-] (D) edge node[left] {} (M);
        \path [-] (M) edge node[left] {} (N);
        \path [-] (O) edge node[left] {} (N);
        \path [-] (O) edge node[left] {} (M);
        \end{tikzpicture}
    \end{subfigure}
    \caption{The graphs $G_0$, $G_1$, and $G_2$.}
    \label{fig:sierp_constr-1}
\end{figure}
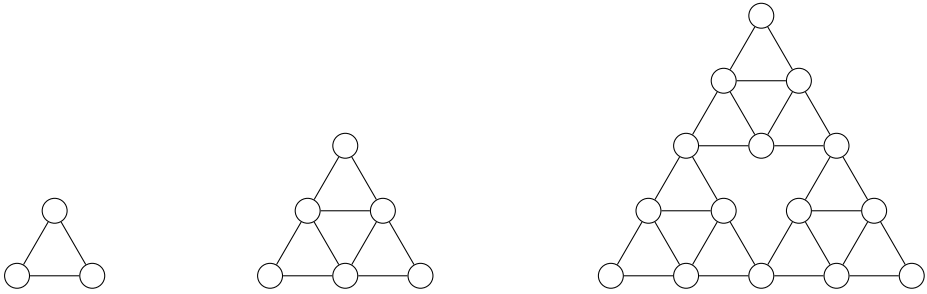
We finally introduce Sierpi\'nski gasket graphs and the associated pre-fractal graphs. The Sierpi\'nski gasket $K$ can be defined formally by the following three similitudes $\psi_i:\mathbb{R}^2\to\mathbb{R}^2$, $i\in\{1,2,3\}$:
$$\psi_i(x)=\frac12(x-u_i)+u_i,$$
where $u_0=(0,0)$, $u_1=(1,0)$ and $u_2=\frac12(1,\sqrt{3})$ are the vertices of a unit equilateral triagle in $\R^2$. Let $\mathcal{G}_0$ be the complete graph over the vertex set $V_0=\{u_0,u_1,u_2\}$. The Sierpi\'nski gasket fractal $K$ is the unique nonempty compact set $K$ such that $K=\cup_{i=1}^3\psi_i(K)$,
whose discrete time approximations are constructed inductively as follows. With $\mathcal{G}_0$ as above, for every $n\geq 1$, we define the associated level-$n$ prefractal graph $\mathcal{G}_n$ as $\mathcal{G}_n:=\cup_{i=1}^3\psi_i(\mathcal{G}_{n-1})$. See Figure \ref{fig:sierp_constr_1} for a graphical representation. In order to make all edges have length one, for any $n\geq 0$ we consider
$G_n:=2^n\mathcal{G}_n,$
where for any $x>0$ and $S\subset \R^2$, $xS:=\{xs:\ s\in S\}$; see Figure \ref{fig:sierp_constr-1}. The one-sided infinite Sierpi\'nski gasket graph $G$ is then defined as the graph $G=\cup_{n=0}^{\infty}G_n$ and the double-sided gasket is defined as $G\cup \mathsf{Refl}(G)$, where $\mathsf{Refl}(G)$ is the reflection of $G$ around  the $y$-axis.

During this work we consider the finite graphs $G_n$, and notice that $G_{n+1}$ is an amalgam of three copies of $G_n$, or more generally 
each $G_n$ is an amalgam of $3^{n-k}$ copies of $G_k$ ($0\leq k\leq n$). The self-similar structure and the fact that three copies of $G_{n-1}$ are matched in three points in order to construct $G_n$ , allows one to solve many problems exactly. For instance, it is known that the number of vertices $|V_n|$ and edges $|E_n|$ of $G_n=(V_n,E_n)$ is given by
$$|V_n|=\frac32(3^n+1)\quad\text{and} \quad |E_n|=3^{n+1}.$$
Also, the number $\tau(G_n)$ of spanning trees of $G_n$ is precisely known and given by the formula
\begin{equation}\label{eq:span-trees-leveln}
\tau(G_n)=\left(\frac{3}{20}\right)^{1/4}\cdot \left(\frac35\right)^{n/2}\cdot 540^{3^n/4}=\left(\frac{5}{12}\right)^{1/4}|E_n|^{(1-2/d_s)/2}\cdot \sqrt[12]{540}^{|E_n|},
\end{equation} 
where $d_s$ is the fractal dimension of the Sierpi\'nski gasket and is given by
$d_s=2\frac{\log 3}{\log 5}\approx 0.86$. This formula has been obtained by different methods in several works; see \cite{teufl-wagner-2011} for more details.
So 
\begin{equation}\label{eq:sp-tr-rec}
\tau(G_{n+1})=\tau(G_n) \cdot 18\cdot 540^{(3^n-1)/2}.
\end{equation}


\section{Main results}\label{sec:results}

\begin{figure}
    \centering
    \begin{subfigure}[t]{0.22\textwidth}
        \includegraphics[width=\linewidth]{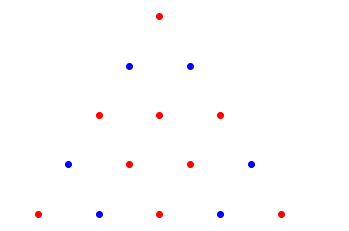}
        \caption{$\mathsf{id}_2$}
    \end{subfigure}
    \begin{subfigure}[t]{0.25\textwidth}
        \includegraphics[width=\linewidth]{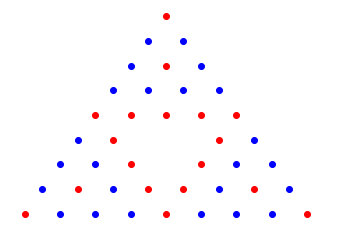}
        \caption{$\mathsf{id}_3$}
    \end{subfigure}
    \begin{subfigure}[t]{0.25\textwidth}
        \includegraphics[width=\linewidth]{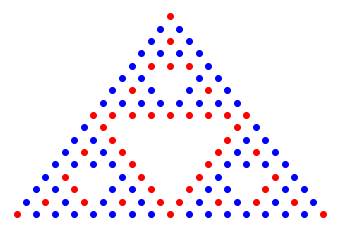}
        \caption{$\mathsf{id}_4$}
    \end{subfigure}
    \begin{subfigure}[t]{0.25\textwidth}
        \includegraphics[width=\linewidth]{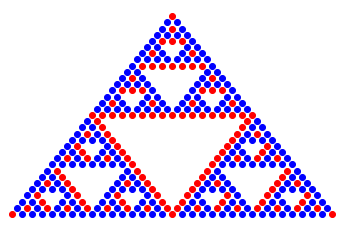}
        \caption{$\mathsf{id}_5$}
    \end{subfigure}
    \caption{The identity element on the first four levels of the Sierpi\'nski gasket. Red dots correspond to vertices with $2$ chips, whereas blue dots correspond to vertices with $3$ chips.}
    \label{fig:identity}
\end{figure}

\subsection{Identity element with normal boundary conditions}

When working with sandpiles, two types of boundary conditions have been considered in the literature:
\vspace{-0.3cm}
\begin{itemize}
\setlength\itemsep{0em}
\item {\it Normal boundary conditions} as in the physics community \cite{waves-sandpile-daerden,crit-exp-daerden-1998} where, each of the three boundary corners of the gasket are connected to a sink vertex $s$ by two additional edges as in Figure \ref{fig:sierp_constr}, so the sink has degree six. In this way, the whole gasket $G_n$, $n\in\mathbb{N}$ contains only non-sink vertices, all with degree four. The advantage in using normal boundary conditions is that all three corner vertices "look" the same, so we do not have to distinguish between vertices  with different degrees and the scaling-invariance and self-similarity can be fully explored.
\item {\it Sinked boundary} where one or more vertices of the underlying graph are collapsed to form a single sink vertex, without collapsing edges.
\end{itemize}
\vspace{-0.3cm}
In \cite{stricharz-sandpile}, for normal boundary conditions the identity element of the sandpile group of the 
Sierpi\'nski gasket cell graphs (or tower of Hanoi graphs) was characterized.  Sierpi\'nski gasket cell graphs represent another class of finite self-similar graphs that can be used to approximate the fractal $K$. Since in obtaining iteration $n$
of such graph, one takes three copies of the previous iteration $n-1$ and joines them by an additional edge, the identity element is easy to characterize, in particular it is shown in \cite{stricharz-sandpile} that the identity element is the constant configuration $2$; in the same paper a graphical conjecture concerning the identity element of Sierpi\'nski gasket graphs $G_n$, for normal boundary conditions is given. We characterize here the identity element of $G_{n}$, for $n\in \N$ and normal boundary conditions.

We emphasize that the case of sinked boundary conditions has been considered in \cite{chen-sandpile-2020}, where the authors characterized the identity elements and gave several toppling identities in the case where the sink is one corner or two corners identified to build the sink. Due to the spatial symmetry of the gasket, it does not matter which corner or two corners are chosen as the sink.  We extend their method to normal boundary conditions and also borrow ideas from their graphical representation of toppling identities and identity element of the sandpile group.

\vspace{-0.3cm}
\paragraph{Some conventions and notations.}The underlying state space will be, for the rest of the paper, the graph $G_n=(V_n\cup\{s\}, E_n)$, $n\in\N$ the $n$-th level of the  Sierpi\'nski gasket graph with normal boundary conditions as in Figure \ref{fig:sierp_constr}, and the sandpile configurations are graphically represented by simply writing the number of chips at the corresponding vertices in a circle. The sink will not be mentioned in all the graphical representations, but we shall always have in mind the fact that from each of the three corner vertices there are two edges going to the same sink. Empty circles indicate an arbitrary amount of chips. To simplify notation, we also write $\mathcal{R}_n$ for the sandpile group of $G_n$, that is $\Rn:=\mathcal{R}(G_n)$.

 If on $G_n$, the left corner $o$ (filled in black) is chosen as a sink vertex, then by adding to $\eta$, $3^n$ chips to each of the remaining two corners and stabilizing results again in $\eta$; this is the claim \cite[Proposition 3.8]{chen-sandpile-2020}. Graphically this can be represented as follows.
\begin{center}
\begin{tikzpicture}[baseline=4.5ex]
    \node[shape=circle,fill=black, draw=black] (A) at (0,0) {};
    \node[shape=circle,draw=black] (B) at (2,0) {$3^n$};
    \node[shape=circle,draw=black] (D) at (1,1.73) {$3^n$};
    
    \path [-] (A) edge node[left] {} (B);
    \path [-] (A) edge node[left] {} (D);
    \path [-] (B) edge node[left] {} (D);
    
    \node[] at (1,1.73/3) {$\eta$};
\end{tikzpicture}
$\rightarrow$
\begin{tikzpicture}[baseline=4.5ex]
    \node[shape=circle,draw=black] (A) at (0,0) {$2\cdot 3^n$};
    \node[shape=circle,draw=black] (B) at (2,0) {};
    \node[shape=circle,draw=black] (D) at (1,1.73) {};
    
    \path [-] (A) edge node[left] {} (B);
    \path [-] (A) edge node[left] {} (D);
    \path [-] (B) edge node[left] {} (D);
    
    \node[] at (1,1.73/3) {$\eta$};
\end{tikzpicture}
\end{center}

\begin{definition}\label{def-m1} We define the sandpile configuration $M_1(x,y,z)$ on $G_1$ by setting the values at the inner vertices as below, and corner values $x,y,z\in\N$ are arbitrary: 
\begin{center}
\begin{tikzpicture}[baseline=7ex, scale=0.7]
    \node[shape=circle,draw=black] (A) at (0,0) {x};
    \node[shape=circle,draw=black] (B) at (2,0) {3};
    \node[shape=circle,draw=black] (C) at (4,0) {y};
    \node[shape=circle,draw=black] (D) at (1,1.73) {3};
    \node[shape=circle,draw=black] (E) at (3,1.73) {2};
    \node[shape=circle,draw=black] (F) at (2,1.73*2) {z} ;
    
    \path [-] (A) edge node[left] {} (B);
    \path [-] (A) edge node[left] {} (D);
    \path [-] (B) edge node[left] {} (D);
    \path [-] (B) edge node[left] {} (C);
    \path [-] (E) edge node[left] {} (B);
    \path [-] (E) edge node[left] {} (C);
    \path [-] (F) edge node[left] {} (E);
    \path [-] (F) edge node[left] {} (D);
    \path [-] (E) edge node[left] {} (D);
\end{tikzpicture}=
\begin{tikzpicture}[baseline=4.5ex]
    \node[shape=circle,draw=black] (A) at (0,0) {x};
    \node[shape=circle,draw=black] (B) at (2,0) {y};
    \node[shape=circle,draw=black] (D) at (1,1.73) {z};
    
    \path [-] (A) edge node[left] {} (B);
    \path [-] (A) edge node[left] {} (D);
    \path [-] (B) edge node[left] {} (D);
    
    \node[] at (1,1.73/3) {$M_1$};
\end{tikzpicture}=$\quad M_1(x,y,z)$.
\end{center}
 For $n\geq 1$, we iteratively define the sandpile configuration $M_{n+1}(x,y,z)$ with boundary values $x,y,z$ on $G_{n+1}$ by setting it equal to $M_n(x,3,3)$ in the lower left triangle, $M_n(3,y,2)$ in the lower right triangle and to $M_n(3,2,z)$ in the upper triangle: 
\begin{center}
\begin{tikzpicture}[baseline=6ex, scale=0.7]
    \node[shape=circle,draw=black] (A) at (0,0) {x};
    \node[shape=circle,draw=black] (B) at (2,0) {3};
    \node[shape=circle,draw=black] (C) at (4,0) {y};
    \node[shape=circle,draw=black] (D) at (1,1.73) {3};
    \node[shape=circle,draw=black] (E) at (3,1.73) {2};
    \node[shape=circle,draw=black] (F) at (2,1.73*2) {z} ;
    
    \path [-] (A) edge node[left] {} (B);
    \path [-] (A) edge node[left] {} (D);
    \path [-] (B) edge node[left] {} (D);
    \path [-] (B) edge node[left] {} (C);
    \path [-] (E) edge node[left] {} (B);
    \path [-] (E) edge node[left] {} (C);
    \path [-] (F) edge node[left] {} (E);
    \path [-] (F) edge node[left] {} (D);
    \path [-] (E) edge node[left] {} (D);
    
    \node[] at (1,1.73/3) {$M_n$};
    \node[] at (3,1.73/3) {$M_n$};
    \node[] at (2,1.73/3+1.73) {$M_n$};
\end{tikzpicture}=
\begin{tikzpicture}[baseline=4ex]
    \node[shape=circle,draw=black] (A) at (0,0) {x};
    \node[shape=circle,draw=black] (B) at (2,0) {y};
    \node[shape=circle,draw=black] (D) at (1,1.73) {z};
    
    \path [-] (A) edge node[left] {} (B);
    \path [-] (A) edge node[left] {} (D);
    \path [-] (B) edge node[left] {} (D);
    
    \node[] at (1,1.73/3) {$M_{n+1}$};
\end{tikzpicture}=$\quad M_{n+1}(x,y,z)$
\end{center}
\end{definition}

We consider first $M_1(2,1,1)$, add it to itself and stabilize all but the lower corner vertex, and investigate the patterns that appear during the stabilization. 
\begin{center}
$M_1(2,1,1)=$
    \begin{tikzpicture}[baseline=4ex, scale=0.6]
    \node[shape=circle,draw=black] (A) at (0,0) {2};
    \node[shape=circle,draw=black] (B) at (2,0) {3};
    \node[shape=circle,draw=black] (C) at (4,0) {1};
    \node[shape=circle,draw=black] (D) at (1,1.73) {3};
    \node[shape=circle,draw=black] (E) at (3,1.73) {2};
    \node[shape=circle,draw=black] (F) at (2,1.73*2) {1} ;
    
    \path [-] (A) edge node[left] {} (B);
    \path [-] (A) edge node[left] {} (D);
    \path [-] (B) edge node[left] {} (D);
    \path [-] (B) edge node[left] {} (C);
    \path [-] (E) edge node[left] {} (B);
    \path [-] (E) edge node[left] {} (C);
    \path [-] (F) edge node[left] {} (E);
    \path [-] (F) edge node[left] {} (D);
    \path [-] (E) edge node[left] {} (D);
    
    \node[] at (1,1.73/3) {};
    \node[] at (3,1.73/3) {};
    \node[] at (2,1.73/3+1.73) {};
\end{tikzpicture}
\end{center}
We first topple once the inner vertices, which results in the new sandpile configuration
\begin{center}
\begin{tikzpicture}[baseline=4.5ex]
    \node[shape=circle,draw=black] (A) at (0,0) {4};
    \node[shape=circle,draw=black] (B) at (2/2,0) {6};
    \node[shape=circle,draw=black] (C) at (4/2,0) {2};
    \node[shape=circle,draw=black] (D) at (1/2,1.73/2) {6};
    \node[shape=circle,draw=black] (E) at (3/2,1.73/2) {4};
    \node[shape=circle,draw=black] (F) at (2/2,1.73) {2} ;
    
    \path [-] (A) edge node[left] {} (B);
    \path [-] (A) edge node[left] {} (D);
    \path [-] (B) edge node[left] {} (D);
    \path [-] (B) edge node[left] {} (C);
    \path [-] (E) edge node[left] {} (B);
    \path [-] (E) edge node[left] {} (C);
    \path [-] (F) edge node[left] {} (E);
    \path [-] (F) edge node[left] {} (D);
    \path [-] (E) edge node[left] {} (D);
    
    \node[] at (1,1.73/3) {};
    \node[] at (3,1.73/3) {};
    \node[] at (2,1.73/3+1.73) {};
\end{tikzpicture}
$\rightarrow$
    \begin{tikzpicture}[baseline=4.5ex]
    \node[shape=circle,draw=black] (A) at (0,0) {6};
    \node[shape=circle,draw=black] (B) at (2/2,0) {3};
    \node[shape=circle,draw=black] (C) at (4/2,0) {3};
    \node[shape=circle,draw=black] (D) at (1/2,1.73/2) {3};
    \node[shape=circle,draw=black] (E) at (3/2,1.73/2) {6};
    \node[shape=circle,draw=black] (F) at (2/2,1.73) {3} ;
    
    \path [-] (A) edge node[left] {} (B);
    \path [-] (A) edge node[left] {} (D);
    \path [-] (B) edge node[left] {} (D);
    \path [-] (B) edge node[left] {} (C);
    \path [-] (E) edge node[left] {} (B);
    \path [-] (E) edge node[left] {} (C);
    \path [-] (F) edge node[left] {} (E);
    \path [-] (F) edge node[left] {} (D);
    \path [-] (E) edge node[left] {} (D);
    
    \node[] at (1,1.73/3) {};
    \node[] at (3,1.73/3) {};
    \node[] at (2,1.73/3+1.73) {};
\end{tikzpicture}
$\rightarrow$
    \begin{tikzpicture}[baseline=4.5ex]
    \node[shape=circle,draw=black] (A) at (0,0) {6};
    \node[shape=circle,draw=black] (B) at (2/2,0) {4};
    \node[shape=circle,draw=black] (C) at (4/2,0) {4};
    \node[shape=circle,draw=black] (D) at (1/2,1.73/2) {4};
    \node[shape=circle,draw=black] (E) at (3/2,1.73/2) {2};
    \node[shape=circle,draw=black] (F) at (2/2,1.73) {4} ;
    
    \path [-] (A) edge node[left] {} (B);
    \path [-] (A) edge node[left] {} (D);
    \path [-] (B) edge node[left] {} (D);
    \path [-] (B) edge node[left] {} (C);
    \path [-] (E) edge node[left] {} (B);
    \path [-] (E) edge node[left] {} (C);
    \path [-] (F) edge node[left] {} (E);
    \path [-] (F) edge node[left] {} (D);
    \path [-] (E) edge node[left] {} (D);
    
    \node[] at (1,1.73/3) {};
    \node[] at (3,1.73/3) {};
    \node[] at (2,1.73/3+1.73) {};
\end{tikzpicture}
\end{center}
We continue by stabilizing first vertices with height $4$, and then continue with the rest, but we do not topple the lower left corner:
\begin{center}
$\rightarrow$
\begin{tikzpicture}[baseline=4.5ex]
    \node[shape=circle,draw=black] (A) at (0,0) {6};
    \node[shape=circle,draw=black] (B) at (2/2,0) {6};
    \node[shape=circle,draw=black] (C) at (4/2,0) {0};
    \node[shape=circle,draw=black] (D) at (1/2,1.73/2) {6};
    \node[shape=circle,draw=black] (E) at (3/2,1.73/2) {6};
    \node[shape=circle,draw=black] (F) at (2/2,1.73) {0} ;
    
    \path [-] (A) edge node[left] {} (B);
    \path [-] (A) edge node[left] {} (D);
    \path [-] (B) edge node[left] {} (D);
    \path [-] (B) edge node[left] {} (C);
    \path [-] (E) edge node[left] {} (B);
    \path [-] (E) edge node[left] {} (C);
    \path [-] (F) edge node[left] {} (E);
    \path [-] (F) edge node[left] {} (D);
    \path [-] (E) edge node[left] {} (D);
    
    \node[] at (1,1.73/3) {};
    \node[] at (3,1.73/3) {};
    \node[] at (2,1.73/3+1.73) {};
\end{tikzpicture}
$\rightarrow$
\begin{tikzpicture}[baseline=4.5ex]
    \node[shape=circle,draw=black] (A) at (0,0) {8};
    \node[shape=circle,draw=black] (B) at (2/2,0) {4};
    \node[shape=circle,draw=black] (C) at (4/2,0) {2};
    \node[shape=circle,draw=black] (D) at (1/2,1.73/2) {4};
    \node[shape=circle,draw=black] (E) at (3/2,1.73/2) {4};
    \node[shape=circle,draw=black] (F) at (2/2,1.73) {2} ;
    
    \path [-] (A) edge node[left] {} (B);
    \path [-] (A) edge node[left] {} (D);
    \path [-] (B) edge node[left] {} (D);
    \path [-] (B) edge node[left] {} (C);
    \path [-] (E) edge node[left] {} (B);
    \path [-] (E) edge node[left] {} (C);
    \path [-] (F) edge node[left] {} (E);
    \path [-] (F) edge node[left] {} (D);
    \path [-] (E) edge node[left] {} (D);
    
    \node[] at (1,1.73/3) {};
    \node[] at (3,1.73/3) {};
    \node[] at (2,1.73/3+1.73) {};
\end{tikzpicture}
$\rightarrow$
\begin{tikzpicture}[baseline=4.5ex]
    \node[shape=circle,draw=black] (A) at (0,0) {10};
    \node[shape=circle,draw=black] (B) at (2/2,0) {2};
    \node[shape=circle,draw=black] (C) at (4/2,0) {4};
    \node[shape=circle,draw=black] (D) at (1/2,1.73/2) {2};
    \node[shape=circle,draw=black] (E) at (3/2,1.73/2) {2};
    \node[shape=circle,draw=black] (F) at (2/2,1.73) {4} ;
    
    \path [-] (A) edge node[left] {} (B);
    \path [-] (A) edge node[left] {} (D);
    \path [-] (B) edge node[left] {} (D);
    \path [-] (B) edge node[left] {} (C);
    \path [-] (E) edge node[left] {} (B);
    \path [-] (E) edge node[left] {} (C);
    \path [-] (F) edge node[left] {} (E);
    \path [-] (F) edge node[left] {} (D);
    \path [-] (E) edge node[left] {} (D);
    
    \node[] at (1,1.73/3) {};
    \node[] at (3,1.73/3) {};
    \node[] at (2,1.73/3+1.73) {};
\end{tikzpicture}
$\rightarrow$
\begin{tikzpicture}[baseline=4.5ex]
    \node[shape=circle,draw=black] (A) at (0,0) {10};
    \node[shape=circle,draw=black] (B) at (2/2,0) {4};
    \node[shape=circle,draw=black] (C) at (4/2,0) {0};
    \node[shape=circle,draw=black] (D) at (1/2,1.73/2) {4};
    \node[shape=circle,draw=black] (E) at (3/2,1.73/2) {6};
    \node[shape=circle,draw=black] (F) at (2/2,1.73) {0} ;
    
    \path [-] (A) edge node[left] {} (B);
    \path [-] (A) edge node[left] {} (D);
    \path [-] (B) edge node[left] {} (D);
    \path [-] (B) edge node[left] {} (C);
    \path [-] (E) edge node[left] {} (B);
    \path [-] (E) edge node[left] {} (C);
    \path [-] (F) edge node[left] {} (E);
    \path [-] (F) edge node[left] {} (D);
    \path [-] (E) edge node[left] {} (D);
    
    \node[] at (1,1.73/3) {};
    \node[] at (3,1.73/3) {};
    \node[] at (2,1.73/3+1.73) {};
\end{tikzpicture}
$\rightarrow$
\begin{tikzpicture}[baseline=4.5ex]
    \node[shape=circle,draw=black] (A) at (0,0) {12};
    \node[shape=circle,draw=black] (B) at (2/2,0) {2};
    \node[shape=circle,draw=black] (C) at (4/2,0) {2};
    \node[shape=circle,draw=black] (D) at (1/2,1.73/2) {2};
    \node[shape=circle,draw=black] (E) at (3/2,1.73/2) {4};
    \node[shape=circle,draw=black] (F) at (2/2,1.73) {2} ;
    
    \path [-] (A) edge node[left] {} (B);
    \path [-] (A) edge node[left] {} (D);
    \path [-] (B) edge node[left] {} (D);
    \path [-] (B) edge node[left] {} (C);
    \path [-] (E) edge node[left] {} (B);
    \path [-] (E) edge node[left] {} (C);
    \path [-] (F) edge node[left] {} (E);
    \path [-] (F) edge node[left] {} (D);
    \path [-] (E) edge node[left] {} (D);
    
    \node[] at (1,1.73/3) {};
    \node[] at (3,1.73/3) {};
    \node[] at (2,1.73/3+1.73) {};
\end{tikzpicture}
$\rightarrow$
\begin{tikzpicture}[baseline=4.5ex]
    \node[shape=circle,draw=black] (A) at (0,0) {12};
    \node[shape=circle,draw=black] (B) at (2/2,0) {4};
    \node[shape=circle,draw=black] (C) at (4/2,0) {1};
    \node[shape=circle,draw=black] (D) at (1/2,1.73/2) {4};
    \node[shape=circle,draw=black] (E) at (3/2,1.73/2) {2};
    \node[shape=circle,draw=black] (F) at (2/2,1.73) {1} ;
    
    \path [-] (A) edge node[left] {} (B);
    \path [-] (A) edge node[left] {} (D);
    \path [-] (B) edge node[left] {} (D);
    \path [-] (B) edge node[left] {} (C);
    \path [-] (E) edge node[left] {} (B);
    \path [-] (E) edge node[left] {} (C);
    \path [-] (F) edge node[left] {} (E);
    \path [-] (F) edge node[left] {} (D);
    \path [-] (E) edge node[left] {} (D);
    
    \node[] at (1,1.73/3) {};
    \node[] at (3,1.73/3) {};
    \node[] at (2,1.73/3+1.73) {};
\end{tikzpicture}
$\rightarrow$
\begin{tikzpicture}[baseline=4.5ex]
    \node[shape=circle,draw=black] (A) at (0,0) {14};
    \node[shape=circle,draw=black] (B) at (2/2,0) {1};
    \node[shape=circle,draw=black] (C) at (4/2,0) {2};
    \node[shape=circle,draw=black] (D) at (1/2,1.73/2) {1};
    \node[shape=circle,draw=black] (E) at (3/2,1.73/2) {4};
    \node[shape=circle,draw=black] (F) at (2/2,1.73) {2} ;
    
    \path [-] (A) edge node[left] {} (B);
    \path [-] (A) edge node[left] {} (D);
    \path [-] (B) edge node[left] {} (D);
    \path [-] (B) edge node[left] {} (C);
    \path [-] (E) edge node[left] {} (B);
    \path [-] (E) edge node[left] {} (C);
    \path [-] (F) edge node[left] {} (E);
    \path [-] (F) edge node[left] {} (D);
    \path [-] (E) edge node[left] {} (D);
    
    \node[] at (1,1.73/3) {};
    \node[] at (3,1.73/3) {};
    \node[] at (2,1.73/3+1.73) {};
\end{tikzpicture}
$\rightarrow$
\begin{tikzpicture}[baseline=4.5ex]
    \node[shape=circle,draw=black] (A) at (0,0) {14};
    \node[shape=circle,draw=black] (B) at (2/2,0) {3};
    \node[shape=circle,draw=black] (C) at (4/2,0) {1};
    \node[shape=circle,draw=black] (D) at (1/2,1.73/2) {3};
    \node[shape=circle,draw=black] (E) at (3/2,1.73/2) {2};
    \node[shape=circle,draw=black] (F) at (2/2,1.73) {1} ;
    
    \path [-] (A) edge node[left] {} (B);
    \path [-] (A) edge node[left] {} (D);
    \path [-] (B) edge node[left] {} (D);
    \path [-] (B) edge node[left] {} (C);
    \path [-] (E) edge node[left] {} (B);
    \path [-] (E) edge node[left] {} (C);
    \path [-] (F) edge node[left] {} (E);
    \path [-] (F) edge node[left] {} (D);
    \path [-] (E) edge node[left] {} (D);
    
    \node[] at (1,1.73/3) {};
    \node[] at (3,1.73/3) {};
    \node[] at (2,1.73/3+1.73) {};
\end{tikzpicture}
\end{center}
Now all but the lower left corner vertex are stable, and we have added $4\cdot 3^1-2=10$ chips to the lower left corner vertex during stabilization of the other vertices, that is, we have obtained the configuration $M_1(2+4\cdot 3^1,1,1)$. Recursively, take now  $M_2(2,1,1)$, add it to itself, and stabilize all the vertices with the exception of the lower left corner:
\begin{center}
    \begin{tikzpicture}[baseline=6ex, scale=0.7]
    \node[shape=circle,draw=black] (A) at (0,0) {4};
    \node[shape=circle,draw=black] (B) at (2,0) {6};
    \node[shape=circle,draw=black] (C) at (4,0) {2};
    \node[shape=circle,draw=black] (D) at (1,1.73) {6};
    \node[shape=circle,draw=black] (E) at (3,1.73) {4};
    \node[shape=circle,draw=black] (F) at (2,1.73*2) {2} ;
    
    \path [-] (A) edge node[left] {} (B);
    \path [-] (A) edge node[left] {} (D);
    \path [-] (B) edge node[left] {} (D);
    \path [-] (B) edge node[left] {} (C);
    \path [-] (E) edge node[left] {} (B);
    \path [-] (E) edge node[left] {} (C);
    \path [-] (F) edge node[left] {} (E);
    \path [-] (F) edge node[left] {} (D);
    \path [-] (E) edge node[left] {} (D);
    
    \node[] at (1,1.73/3) {$2M_1$};
    \node[] at (3,1.73/3) {$2M_1$};
    \node[] at (2,1.73/3+1.73) {$2M_1$};
\end{tikzpicture}
$\rightarrow$
\begin{tikzpicture}[baseline=7ex, scale=0.8]
    \node[shape=circle,draw=black] (A) at (0,0) {4};
    \node[shape=circle,draw=black] (B) at (2,0) {16};
    \node[shape=circle,draw=black] (C) at (4,0) {1};
    \node[shape=circle,draw=black] (D) at (1,1.73) {16};
    \node[shape=circle,draw=black] (E) at (3,1.73) {2};
    \node[shape=circle,draw=black] (F) at (2,1.73*2) {1} ;
    
    \path [-] (A) edge node[left] {} (B);
    \path [-] (A) edge node[left] {} (D);
    \path [-] (B) edge node[left] {} (D);
    \path [-] (B) edge node[left] {} (C);
    \path [-] (E) edge node[left] {} (B);
    \path [-] (E) edge node[left] {} (C);
    \path [-] (F) edge node[left] {} (E);
    \path [-] (F) edge node[left] {} (D);
    \path [-] (E) edge node[left] {} (D);
    
    \node[] at (1,1.73/3) {$2M_1$};
    \node[] at (3,1.73/3) {$M_1$};
    \node[] at (2,1.73/3+1.73) {$M_1$};
\end{tikzpicture}
$\rightarrow$
\begin{tikzpicture}[baseline=7ex, scale=0.8]
    \node[shape=circle,draw=black] (A) at (0,0) {28};
    \node[shape=circle,draw=black] (B) at (2,0) {4};
    \node[shape=circle,draw=black] (C) at (4,0) {1};
    \node[shape=circle,draw=black] (D) at (1,1.73) {4};
    \node[shape=circle,draw=black] (E) at (3,1.73) {2};
    \node[shape=circle,draw=black] (F) at (2,1.73*2) {1} ;
    
    \path [-] (A) edge node[left] {} (B);
    \path [-] (A) edge node[left] {} (D);
    \path [-] (B) edge node[left] {} (D);
    \path [-] (B) edge node[left] {} (C);
    \path [-] (E) edge node[left] {} (B);
    \path [-] (E) edge node[left] {} (C);
    \path [-] (F) edge node[left] {} (E);
    \path [-] (F) edge node[left] {} (D);
    \path [-] (E) edge node[left] {} (D);
    
    \node[] at (1,1.73/3) {$2M_1$};
    \node[] at (3,1.73/3) {$M_1$};
    \node[] at (2,1.73/3+1.73) {$M_1$};
\end{tikzpicture}
$\rightarrow$
\hspace{1cm}
\begin{tikzpicture}[baseline=6ex, scale=0.7]
    \node[shape=circle,draw=black] (A) at (0,0) {38};
    \node[shape=circle,draw=black] (B) at (2,0) {3};
    \node[shape=circle,draw=black] (C) at (4,0) {1};
    \node[shape=circle,draw=black] (D) at (1,1.73) {3};
    \node[shape=circle,draw=black] (E) at (3,1.73) {2};
    \node[shape=circle,draw=black] (F) at (2,1.73*2) {1} ;
    
    \path [-] (A) edge node[left] {} (B);
    \path [-] (A) edge node[left] {} (D);
    \path [-] (B) edge node[left] {} (D);
    \path [-] (B) edge node[left] {} (C);
    \path [-] (E) edge node[left] {} (B);
    \path [-] (E) edge node[left] {} (C);
    \path [-] (F) edge node[left] {} (E);
    \path [-] (F) edge node[left] {} (D);
    \path [-] (E) edge node[left] {} (D);
    
    \node[] at (1,1.73/3) {$M_1$};
    \node[] at (3,1.73/3) {$M_1$};
    \node[] at (2,1.73/3+1.73) {$M_1$};
\end{tikzpicture}
\end{center}
In the third stabilization procedure we have used  \cite[Proposition 3.8]{chen-sandpile-2020}, and the lower corner vertex collected $34=4\cdot 3^2-2$ particles during stabilization of the other vertices, and the resulting configuration is $M_2(2+4\cdot 3^2)$. Then, the inductive step follows immediately in the same way as in the proof of \cite[Proposition 3.8]{chen-sandpile-2020}; thus on $G_n$, we have shown that
starting with the configuration $M_n(2,1,1)$ adding to itself and stabilizing all but the lower left corner, one adds $4\cdot 3^n-2$ particles to the lower left corner during stabilization and reaches the configuration $M_n(2+4\cdot 3^n,1,1)$ which is still not stable when considering normal boundary conditions, therefore
\begin{equation}\label{eq:2mn}
\left(2M_n(2,1,1)\right)^{\circ}=\left(M_n(2+4\cdot 3^n,1,1)\right)^{\circ}
\end{equation}
and this may be used in describing the identity element of $G_n$ with normal boundary conditions.

\begin{theorem}\label{thm:id}
Denote by $M_n^+$ (respectively $M_n^-$) the sandpile configuration on $G_n$ obtained from $M_n$ by rotating $G_n$  counterclockwise (respectively clockwise) $120^\circ$. Then, for any $n\geq 1$ the identity element $\mathsf{id}_n$ of the sandpile group $(\Rn,\oplus)$ of $G_n$ with normal boundary conditions is given by:
\begin{center}
    \begin{tikzpicture}[baseline=6ex, scale=0.7]
    \node[shape=circle,draw=black] (A) at (0,0) {2};
    \node[shape=circle,draw=black] (B) at (2,0) {2};
    \node[shape=circle,draw=black] (C) at (4,0) {2};
    \node[shape=circle,draw=black] (D) at (1,1.73) {2};
    \node[shape=circle,draw=black] (E) at (3,1.73) {2};
    \node[shape=circle,draw=black] (F) at (2,1.73*2) {2} ;
    
    \path [-] (A) edge node[left] {} (B);
    \path [-] (A) edge node[left] {} (D);
    \path [-] (B) edge node[left] {} (D);
    \path [-] (B) edge node[left] {} (C);
    \path [-] (E) edge node[left] {} (B);
    \path [-] (E) edge node[left] {} (C);
    \path [-] (F) edge node[left] {} (E);
    \path [-] (F) edge node[left] {} (D);
    \path [-] (E) edge node[left] {} (D);
    
    \node[] at (1,1.73/3) {$M_{n}$};
    \node[] at (3,1.73/3) {$M^+_{n}$};
    \node[] at (2,1.73/3+1.73) {$M^-_{n}$};
\end{tikzpicture}=
$\quad \mathsf{id}_{n+1}$.
\end{center}
\end{theorem}
The proof of Theorem \ref{thm:id} follows by an easy induction argument together with the burning algorithm, by matching three copies of $M_n$, rotated as in the claim, and by setting the sandpile configuration equal to two at cut (or {\it junction}) points. See Figure \ref{fig:identity} for a graphical representation of the identity element of $\mathcal{R}_2,\mathcal{R}_3$, $\mathcal{R}_4$ and $\mathcal{R}_5$.

\begin{proof}[Proof of Theorem \ref{thm:id}]
The fact that $\mathsf{id}_n$ is recurrent can be checked by induction and by applying the burning algorithm. We now prove that $\mathsf{id}_n$ as considered in the claim is indeed the identity of the sandpile group. We do this by proving that $(2\mathsf{id}_{n})^{\circ}=\mathsf{id}_{n}$. In doing so, we first show by induction over $n$ that
for the sequence $f_{n}=4\cdot 3^{n}+2$, for $n\in\N$, if $\mathsf{id}_{n+1}$ is given as in the statement of the theorem,  then it holds
\begin{align*}
    (2\mathsf{id}_{n+1})^\circ=(\mathsf{id}_{n+1}+4\cdot 3^{n}(\delta_1+\delta_2+\delta_3))^\circ
    \end{align*}
where $\delta_1$ is the sandpile configuration that is $1$ at the lower left corner vertex and $0$ elsewhere, $\delta_2$ is the vector that is $1$ at the lower right corner and $0$ elsewhere, and $\delta_3$  is the vector that is $1$ at the upper corner and $0$ elsewhere. Graphically, this means showing that:
\begin{center}
\begin{tikzpicture}[baseline=7ex, scale=0.8]
    \node[shape=circle,draw=black] (A) at (0,0) {4};
    \node[shape=circle,draw=black] (B) at (2,0) {4};
    \node[shape=circle,draw=black] (C) at (4,0) {4};
    \node[shape=circle,draw=black] (D) at (1,1.73) {4};
    \node[shape=circle,draw=black] (E) at (3,1.73) {4};
    \node[shape=circle,draw=black] (F) at (2,1.73*2) {4} ;
    
    \path [-] (A) edge node[left] {} (B);
    \path [-] (A) edge node[left] {} (D);
    \path [-] (B) edge node[left] {} (D);
    \path [-] (B) edge node[left] {} (C);
    \path [-] (E) edge node[left] {} (B);
    \path [-] (E) edge node[left] {} (C);
    \path [-] (F) edge node[left] {} (E);
    \path [-] (F) edge node[left] {} (D);
    \path [-] (E) edge node[left] {} (D);
    
    \node[] at (1,1.73/3) {$2M_{n}$};
    \node[] at (3,1.73/3) {$2M^+_{n}$};
    \node[] at (2,1.73/3+1.73) {$2M^-_{n}$};
\end{tikzpicture}
$\rightarrow$
    \begin{tikzpicture}[baseline=6ex, scale=0.7]
    \node[shape=circle,draw=black] (A) at (0,0) {$f_{n}$};
    \node[shape=circle,draw=black] (B) at (2,0) {2};
    \node[shape=circle,draw=black] (C) at (4,0) {$f_{n}$};
    \node[shape=circle,draw=black] (D) at (1,1.73) {2};
    \node[shape=circle,draw=black] (E) at (3,1.73) {2};
    \node[shape=circle,draw=black] (F) at (2,1.73*2) {$f_{n}$} ;
    
    \path [-] (A) edge node[left] {} (B);
    \path [-] (A) edge node[left] {} (D);
    \path [-] (B) edge node[left] {} (D);
    \path [-] (B) edge node[left] {} (C);
    \path [-] (E) edge node[left] {} (B);
    \path [-] (E) edge node[left] {} (C);
    \path [-] (F) edge node[left] {} (E);
    \path [-] (F) edge node[left] {} (D);
    \path [-] (E) edge node[left] {} (D);
    
    \node[] at (1,1.73/3) {$M_{n}$};
    \node[] at (3,1.73/3) {$M^+_{n}$};
    \node[] at (2,1.73/3+1.73) {$M^-_{n}$};
\end{tikzpicture}
\end{center}
For the induction base $n=1$, it is easy to see that 
\begin{center}
    \begin{tikzpicture}[baseline=6ex, scale=0.7]
    \node[shape=circle,draw=black] (A) at (0,0) {4};
    \node[shape=circle,draw=black] (B) at (2,0) {4};
    \node[shape=circle,draw=black] (C) at (4,0) {4};
    \node[shape=circle,draw=black] (D) at (1,1.73) {4};
    \node[shape=circle,draw=black] (E) at (3,1.73) {4};
    \node[shape=circle,draw=black] (F) at (2,1.73*2) {4} ;
    
    \path [-] (A) edge node[left] {} (B);
    \path [-] (A) edge node[left] {} (D);
    \path [-] (B) edge node[left] {} (D);
    \path [-] (B) edge node[left] {} (C);
    \path [-] (E) edge node[left] {} (B);
    \path [-] (E) edge node[left] {} (C);
    \path [-] (F) edge node[left] {} (E);
    \path [-] (F) edge node[left] {} (D);
    \path [-] (E) edge node[left] {} (D);
    
    \node[] at (1,1.73/3) {$2M_1$};
    \node[] at (3,1.73/3) {$2M_1^+$};
    \node[] at (2,1.73/3+1.73) {$2M_1^-$};
\end{tikzpicture}
$\rightarrow$
    \begin{tikzpicture}[baseline=7ex, scale=0.8]
    \node[shape=circle,draw=black] (A) at (0,0) {14};
    \node[shape=circle,draw=black] (B) at (2,0) {2};
    \node[shape=circle,draw=black] (C) at (4,0) {14};
    \node[shape=circle,draw=black] (D) at (1,1.73) {2};
    \node[shape=circle,draw=black] (E) at (3,1.73) {2};
    \node[shape=circle,draw=black] (F) at (2,1.73*2) {14} ;
    
    \path [-] (A) edge node[left] {} (B);
    \path [-] (A) edge node[left] {} (D);
    \path [-] (B) edge node[left] {} (D);
    \path [-] (B) edge node[left] {} (C);
    \path [-] (E) edge node[left] {} (B);
    \path [-] (E) edge node[left] {} (C);
    \path [-] (F) edge node[left] {} (E);
    \path [-] (F) edge node[left] {} (D);
    \path [-] (E) edge node[left] {} (D);
    
    \node[] at (1,1.73/3) {$M_1$};
    \node[] at (3,1.73/3) {$M_1^+$};
    \node[] at (2,1.73/3+1.73) {$M_1^-$};
\end{tikzpicture}
\end{center}
since the topplings in the three subcopies $G_0$ of $G_1$ interact only through the junction points, and every time a junction point topples it sents two chips in each of the two adjacent triangles, and every time the four neighbors of a junction point topple, the junction point receives two chips from each copy of $G_0$. So performing the topplings on $G_1$ for the given configuration is equivalent to performing the topplings on $G_0$ with half the mass (two instead of four chips) at junction points and using the lower left vertex as the one that collects the mass and does not topple. This is allowed due to the Abelian property of the model. That is, we consider $2M_1(2,1,1)$ and bring the mass to the lower left corner. Doing the same for $M_n^+$ and $M_n^{-}$ in the remaining two triangles, using that $M_n^+$ and $M_n^{-}$  are rotations of $M_n$, and adding the mass at the common junction points gives the induction base:
$$(2\mathsf{id}_2)^{\circ}=(2M_2(2,1,1))^{\circ}=(M_2(2+4\cdot 3^1,1,1))^{\circ}=
(\mathsf{id}_{2}+4\cdot 3^{1}(\delta_1+\delta_2+\delta_3))^\circ=\mathsf{id}_2,$$
where the last equation above follows from Dhar's identity test since each of the three corner vertices is connected by two edges to the sink. 
The inductive step follows in the same way and it can be easily understood graphically as below:

\begin{center}
\begin{tikzpicture}[baseline=6ex]
    \node[shape=circle,draw=black] (A) at (0,0) {4};
    \node[shape=circle,draw=black] (B) at (2,0) {4};
    \node[shape=circle,draw=black] (C) at (4,0) {4};
    \node[shape=circle,draw=black] (D) at (1,1.73) {4};
    \node[shape=circle,draw=black] (E) at (3,1.73) {4};
    \node[shape=circle,draw=black] (F) at (2,1.73*2) {4} ;
    
    \path [-] (A) edge node[left] {} (B);
    \path [-] (A) edge node[left] {} (D);
    \path [-] (B) edge node[left] {} (D);
    \path [-] (B) edge node[left] {} (C);
    \path [-] (E) edge node[left] {} (B);
    \path [-] (E) edge node[left] {} (C);
    \path [-] (F) edge node[left] {} (E);
    \path [-] (F) edge node[left] {} (D);
    \path [-] (E) edge node[left] {} (D);
    
    \node[] at (1,1.73/3) {$2M_n$};
    \node[] at (3,1.73/3) {$2M_n^+$};
    \node[] at (2,1.73/3+1.73) {$2M_n^-$};
\end{tikzpicture}
$\Leftrightarrow$
\begin{tikzpicture}[baseline=6ex]
    \node[shape=circle,draw=black] (A) at (0,0) {4};
    \node[shape=circle,draw=black] (B) at (2,0) {2};
    \node[shape=circle,draw=black] (G) at (2+1,0) {2};
    \node[shape=circle,draw=black] (C) at (4+1,0) {4};
    \node[shape=circle,draw=black] (D) at (1,1.73) {2};
    \node[shape=circle,draw=black] (E) at (3+1,1.73) {2};
    \node[shape=circle,draw=black] (F) at (2+1/2,1.73*2+1) {4};
    \node[shape=circle,draw=black] (H) at (1+1/2,1.73+1) {2};
    \node[shape=circle,draw=black] (I) at (4-1/2,1.73+1) {2};
    
    \path [-] (A) edge node[left] {} (B);
    \path [-] (A) edge node[left] {} (D);
    \path [-] (B) edge node[left] {} (D);
    \path [-] (G) edge node[left] {} (C);
    \path [-] (G) edge node[left] {} (E);
    \path [-] (E) edge node[left] {} (C);
    \path [-] (F) edge node[left] {} (H);
    \path [-] (F) edge node[left] {} (I);
    \path [-] (I) edge node[left] {} (H);
    
    \draw [dotted] (G) edge node[left] {} (B);
    \draw [dotted] (E) edge node[left] {} (I);
    \draw [dotted] (D) edge node[left] {} (H);
    
    \node[] at (1,1.73/3) {$2M_n$};
    \node[] at (3+1,1.73/3) {$2M_n^+$};
    \node[] at (2+1/2,1.73/3+1.73+1) {$2M_n^-$};
\end{tikzpicture}
$\rightarrow$
\\
$\rightarrow$
\begin{tikzpicture}[baseline=6ex]
    \node[shape=circle,draw=black] (A) at (0,0) {$f_n$};
    \node[shape=circle,draw=black] (B) at (2,0) {1};
    \node[shape=circle,draw=black] (G) at (2+1,0) {1};
    \node[shape=circle,draw=black] (C) at (4+1,0) {$f_n$};
    \node[shape=circle,draw=black] (D) at (1,1.73) {1};
    \node[shape=circle,draw=black] (E) at (3+1,1.73) {1};
    \node[shape=circle,draw=black] (F) at (2+1/2,1.73*2+1) {$f_n$};
    \node[shape=circle,draw=black] (H) at (1+1/2,1.73+1) {1};
    \node[shape=circle,draw=black] (I) at (4-1/2,1.73+1) {1};
    
    \path [-] (A) edge node[left] {} (B);
    \path [-] (A) edge node[left] {} (D);
    \path [-] (B) edge node[left] {} (D);
    \path [-] (G) edge node[left] {} (C);
    \path [-] (G) edge node[left] {} (E);
    \path [-] (E) edge node[left] {} (C);
    \path [-] (F) edge node[left] {} (H);
    \path [-] (F) edge node[left] {} (I);
    \path [-] (I) edge node[left] {} (H);
    
    \draw [dotted] (G) edge node[left] {} (B);
    \draw [dotted] (E) edge node[left] {} (I);
    \draw [dotted] (D) edge node[left] {} (H);
    
    \node[] at (1,1.73/3) {$M_n$};
    \node[] at (3+1,1.73/3) {$M_n^+$};
    \node[] at (2+1/2,1.73/3+1.73+1) {$M_n^-$};
\end{tikzpicture}
$\Leftrightarrow$
\begin{tikzpicture}[baseline=6ex]
    \node[shape=circle,draw=black] (A) at (0,0) {$f_n$};
    \node[shape=circle,draw=black] (B) at (2,0) {2};
    \node[shape=circle,draw=black] (C) at (4,0) {$f_n$};
    \node[shape=circle,draw=black] (D) at (1,1.73) {2};
    \node[shape=circle,draw=black] (E) at (3,1.73) {2};
    \node[shape=circle,draw=black] (F) at (2,1.73*2) {$f_n$} ;
    
    \path [-] (A) edge node[left] {} (B);
    \path [-] (A) edge node[left] {} (D);
    \path [-] (B) edge node[left] {} (D);
    \path [-] (B) edge node[left] {} (C);
    \path [-] (E) edge node[left] {} (B);
    \path [-] (E) edge node[left] {} (C);
    \path [-] (F) edge node[left] {} (E);
    \path [-] (F) edge node[left] {} (D);
    \path [-] (E) edge node[left] {} (D);
    
    \node[] at (1,1.73/3) {$M_n$};
    \node[] at (3,1.73/3) {$M_n^+$};
    \node[] at (2,1.73/3+1.73) {$M_n^-$};
\end{tikzpicture}$\rightarrow$
\\[2\baselineskip]$\rightarrow$
\begin{tikzpicture}[baseline=6ex]
    \node[shape=circle,draw=black] (A) at (0,0) {2};
    \node[shape=circle,draw=black] (B) at (2,0) {2};
    \node[shape=circle,draw=black] (C) at (4,0) {2};
    \node[shape=circle,draw=black] (D) at (1,1.73) {2};
    \node[shape=circle,draw=black] (E) at (3,1.73) {2};
    \node[shape=circle,draw=black] (F) at (2,1.73*2) {2} ;
    
    \path [-] (A) edge node[left] {} (B);
    \path [-] (A) edge node[left] {} (D);
    \path [-] (B) edge node[left] {} (D);
    \path [-] (B) edge node[left] {} (C);
    \path [-] (E) edge node[left] {} (B);
    \path [-] (E) edge node[left] {} (C);
    \path [-] (F) edge node[left] {} (E);
    \path [-] (F) edge node[left] {} (D);
    \path [-] (E) edge node[left] {} (D);
    
    \node[] at (1,1.73/3) {$M_n$};
    \node[] at (3,1.73/3) {$M_n^+$};
    \node[] at (2,1.73/3+1.73) {$M_n^-$};
\end{tikzpicture}
\end{center}
Thus
\begin{align*}
    (2\mathsf{id}_{n+1})^\circ=(\mathsf{id}_{n+1}+4\cdot 3^{n}(\delta_1+\delta_2+\delta_3))^\circ=\mathsf{id}_{n+1},
    \end{align*}
    where once again, the last equation follows from Dhar's identity test and this
shows that $\mathsf{id}_n$ as claimed  is the identity element of the sandpile group $\mathcal{R}_n$ with normal boundary conditions.
\end{proof}
\textbf{Remark.} Many of the toppling identities from \cite{chen-sandpile-2020}
can be extended to normal boundary conditions as in the proof of Theorem \ref{thm:id}. More precisely, we can use \cite[Proposition 3.8]{chen-sandpile-2020} three times, for recurrent configurations $\eta$ and their $120^{\circ}$ (counterclockwise and clockwise) rotations and bring the excess mass to the three boundary corners of the triangle.  Then we match the new configurations together at junction points. More precisely, let $\eta_n\in \Rn$ be any recurrent configuration on $G_n$ with normal boundary conditions, $\eta^+$, $\eta^-$ be its $120^{\circ}$ counterclockwise and clockwise rotations, respectively. Writing $\eta_n(\star,2)$ for a recurrent configuration with boundary values as
\begin{center}
$ \eta_n(\star,2):=$\quad
\begin{tikzpicture}[baseline=4.5ex]
    \node[shape=circle,draw=black] (A) at (0,0) {$\star$};
    \node[shape=circle,draw=black] (B) at (2,0) {$2$};
    \node[shape=circle,draw=black] (D) at (1,1.73) {$2$};
    
    \path [-] (A) edge node[left] {} (B);
    \path [-] (A) edge node[left] {} (D);
    \path [-] (B) edge node[left] {} (D);
    
    \node[] at (1,1.73/3) {$\eta_n$};
\end{tikzpicture}
\end{center}
where $\star$ represents an arbitrary number of chips that makes $\eta$ recurrent, then one can prove by induction over $n$ that the configuration $\eta_{n+1}$ on $G_{n+1}$ defined as 
\begin{center}
$\eta_{n+1}(\star):=$
    \begin{tikzpicture}[baseline=7ex, scale=0.8]
    \node[shape=circle,draw=black] (A) at (0,0) {$\star$};
    \node[shape=circle,draw=black] (B) at (2,0) {$2$};
    \node[shape=circle,draw=black] (C) at (4,0) {$\star$};
    \node[shape=circle,draw=black] (D) at (1,1.73) {$2$};
    \node[shape=circle,draw=black] (E) at (3,1.73) {$2$};
    \node[shape=circle,draw=black] (F) at (2,1.73*2) {$\star$} ;
    
    \path [-] (A) edge node[left] {} (B);
    \path [-] (A) edge node[left] {} (D);
    \path [-] (B) edge node[left] {} (D);
    \path [-] (B) edge node[left] {} (C);
    \path [-] (E) edge node[left] {} (B);
    \path [-] (E) edge node[left] {} (C);
    \path [-] (F) edge node[left] {} (E);
    \path [-] (F) edge node[left] {} (D);
    \path [-] (E) edge node[left] {} (D);
    
    \node[] at (1,1.73/3) {$\eta_{n}$};
    \node[] at (3,1.73/3) {$\eta^+_{n}$};
    \node[] at (2,1.73/3+1.73) {$\eta^-_{n}$};
\end{tikzpicture}
\end{center}
is also recurrent. We leave the details of this calculation to the reader. Moreover, the following  holds:
\begin{center}
 \begin{tikzpicture}[baseline=6ex,scale=0.7]
    \node[shape=circle,draw=black] (A) at (0,0) {$\star$};
    \node[shape=circle,draw=black] (B) at (2,0) {2};
    \node[shape=circle,draw=black] (C) at (4,0) {$\star$};
    \node[shape=circle,draw=black] (D) at (1,1.73) {2};
    \node[shape=circle,draw=black] (E) at (3,1.73) {2};
    \node[shape=circle,draw=black] (F) at (2,1.73*2) {$\star$} ;
    
    \path [-] (A) edge node[left] {} (B);
    \path [-] (A) edge node[left] {} (D);
    \path [-] (B) edge node[left] {} (D);
    \path [-] (B) edge node[left] {} (C);
    \path [-] (E) edge node[left] {} (B);
    \path [-] (E) edge node[left] {} (C);
    \path [-] (F) edge node[left] {} (E);
    \path [-] (F) edge node[left] {} (D);
    \path [-] (E) edge node[left] {} (D);
    
    \node[] at (1,1.73/3) {$\eta_{n}$};
    \node[] at (3,1.73/3) {$\eta^+_{n}$};
    \node[] at (2,1.73/3+1.73) {$\eta^-_{n}$};
\end{tikzpicture}+
 \begin{tikzpicture}[baseline=6ex,scale=0.7]
    \node[shape=circle,draw=black] (A) at (0,0) {$0$};
    \node[shape=rectangle,draw=black] (B) at (2,0) {$2\cdot 3^n$};
    \node[shape=circle,draw=black] (C) at (4,0) {$0$};
    \node[shape=rectangle,draw=black] (D) at (1,1.73) {$2\cdot 3^n$};
    \node[shape=rectangle,draw=black] (E) at (3,1.73) {$2\cdot 3^n$};
    \node[shape=circle,draw=black] (F) at (2,1.73*2) {$0$} ;
    
    \path [-] (A) edge node[left] {} (B);
    \path [-] (A) edge node[left] {} (D);
    \path [-] (B) edge node[left] {} (D);
    \path [-] (B) edge node[left] {} (C);
    \path [-] (E) edge node[left] {} (B);
    \path [-] (E) edge node[left] {} (C);
    \path [-] (F) edge node[left] {} (E);
    \path [-] (F) edge node[left] {} (D);
    \path [-] (E) edge node[left] {} (D);
    
    \node[] at (1,1.73/3) {$\mathbf{0}$};
    \node[] at (3,1.73/3) {$\mathbf{0}$};
    \node[] at (2,1.73/3+1.73) {$\mathbf{0}$};
\end{tikzpicture}
$\rightarrow$
    \begin{tikzpicture}[baseline=6ex, scale=0.9]
    \node[shape=rectangle,draw=black] (A) at (0,0) {$2\cdot 3^n+\star$};
    \node[shape=circle,draw=black] (B) at (2,0) {2};
    \node[shape=rectangle,draw=black] (C) at (4,0) {$2\cdot 3^n+\star$};
    \node[shape=circle,draw=black] (D) at (1,1.73) {2};
    \node[shape=circle,draw=black] (E) at (3,1.73) {2};
    \node[shape=rectangle,draw=black] (F) at (2,1.73*2) {$2\cdot 3^n+\star$} ;
    
    \path [-] (A) edge node[left] {} (B);
    \path [-] (A) edge node[left] {} (D);
    \path [-] (B) edge node[left] {} (D);
    \path [-] (B) edge node[left] {} (C);
    \path [-] (E) edge node[left] {} (B);
    \path [-] (E) edge node[left] {} (C);
    \path [-] (F) edge node[left] {} (E);
    \path [-] (F) edge node[left] {} (D);
    \path [-] (E) edge node[left] {} (D);
    
    \node[] at (1,1.73/3) {$\eta_{n}$};
    \node[] at (3,1.73/3) {$\eta^+_{n}$};
    \node[] at (2,1.73/3+1.73) {$\eta^-_{n}$};
\end{tikzpicture}
\\[2\baselineskip]$\rightarrow$
    \begin{tikzpicture}[baseline=6ex,scale=0.7]
    \node[shape=circle,draw=black] (A) at (0,0) {$\star$};
    \node[shape=circle,draw=black] (B) at (2,0) {2};
    \node[shape=circle,draw=black] (C) at (4,0) {$\star$};
    \node[shape=circle,draw=black] (D) at (1,1.73) {2};
    \node[shape=circle,draw=black] (E) at (3,1.73) {2};
    \node[shape=circle,draw=black] (F) at (2,1.73*2) {$\star$} ;
    
    \path [-] (A) edge node[left] {} (B);
    \path [-] (A) edge node[left] {} (D);
    \path [-] (B) edge node[left] {} (D);
    \path [-] (B) edge node[left] {} (C);
    \path [-] (E) edge node[left] {} (B);
    \path [-] (E) edge node[left] {} (C);
    \path [-] (F) edge node[left] {} (E);
    \path [-] (F) edge node[left] {} (D);
    \path [-] (E) edge node[left] {} (D);
    
    \node[] at (1,1.73/3) {$\eta_{n}$};
    \node[] at (3,1.73/3) {$\eta^+_{n}$};
    \node[] at (2,1.73/3+1.73) {$\eta^-_{n}$};
\end{tikzpicture}
\end{center}
So starting with a recurrent configuration $\eta_n$ on $G_n$ that has two chips at the lower right vertex and at the upper vertex, matching $\eta_n,\eta_n^+$ and $\eta_n^-$ in the rotated fashion, adding additional $2\cdot 3^n$ chips at the three junction points and stabilizing, leaves $\eta$ invariant. This is yet another easy exercise that is left to the reader.

\subsection{Sandpile group $\Rn$ with normal boundary conditions}

The formula \eqref{eq:sp-tr-rec} for the number of spanning trees of $G_n$ and its relation with the sandpile group $\Rn$ raises immediately the question of a similar factorization for $\Rn$. 
Given the recursive construction of $G_n$ obtrained by matching three copies of $G_{n-1}$ at junction points, it seems tempting to assume that this recursive structure carries over to the sandpile group $\mathcal{R}_n$ on $G_n$. However, this is unfortunately not the case. The fact that two copies of $G_{n-1}$ in $G_n$ interact at junction points makes things a bit more subtle. However, by ignoring what happens at and around cut points by modding out the equivalence class of the configurations that are $1$ at a cut point and $0$ everywhere else, we can give a recursive characterization of $\Rn\cong \Z^{V_n}\big\slash\triangle_n\Z^{V_n}:=\Gamma_n$ in terms of $\mathcal{R}_{n-1}$. A similar recursive characterization of the sandpile group was given on trees in \cite{Levine-sandpile-tree} and in \cite{toumpakari-sandpile}.

Given $G_n$, denote by $a_n,b_n$ and $c_n$ the three cut points where the three copies of $G_{n-1}$ meet, and by $x_n,y_n$ and $z_n$ the three corner vertices of $G_n$ as in the Figure \ref{fig:match-gn}. 

Denote by $\mathsf{u}^{\uparrow}\in \Z^{V_n}$ the vector indexed over the set $V_n$ of vertices of $G_n$ which equals $1$ at the two neighbors of $a_n$ in the upper copy of $G_{n-1}$ and $0$ everywhere else, let $\mathsf{u}^{\leftarrow}\in \Z^{V_n}$ be the vector which equals to $1$ at the two neighbors of $c_n$ in the left copy of $G_{n-1}$ and $0$ everywhere else. Finally, denote by $\mathsf{u}^{\rightarrow}\in \Z^{V_n}$ the vector which equals $1$ at the two neighbors of $b_n$ in the right copy (below $b_n$)  of $G_{n-1}$ and $0$ everywhere else.

For $\mathsf{u}\in \Z^{V_n}$, we write  $[\mathsf{u}]$ for the equivalence class of $\mathsf{u}$ and  $\langle \mathsf{u} \rangle $ for the cyclic subgroup of $\Gamma_n$ generated by $[\mathsf{u}]$. Then we have the following characterization.
\begin{theorem}\label{thm:sand-group}
For the sandpile group $\Rn\cong \Z^{V_n}\big\slash\triangle_n\Z^{V_n}=\Gamma_n$ of $G_n$ with normal boundary conditions, for every $n\in \mathbb{N}$ , we have the following isomorphism
\begin{align*}
    \Gamma_n\Big\slash\Big\langle[\mathsf{u}^{\uparrow}],[\mathsf{u}^{\leftarrow}],[\mathsf{u}^{\rightarrow}],[\delta_{a_n}],[\delta_{b_n}],[\delta_{c_n}]\Big\rangle
    \cong \Gamma^{\uparrow}_{n-1}\oplus \Gamma^{\leftarrow}_{n-1} \oplus \Gamma^{\rightarrow}_{n-1},
    \end{align*}
where $\triangle_n$ denotes the reduced  Laplacian of $G_n$ and 
\begin{align*}
\Gamma^{\uparrow}_{n-1}&:=\Gamma_{n-1}\big\slash\langle[\delta_{x_{n-1}}],[\delta_{y_{n-1}}]\rangle\\
\Gamma^{\leftarrow}_{n-1}&:=\Gamma_{n-1}\big\slash\langle[\delta_{y_{n-1}}],[\delta_{z_{n-1}}]\rangle\\
\Gamma^{\rightarrow}_{n-1}&:=\Gamma_{n-1}\big\slash\langle[\delta_{z_{n-1}}],[\delta_{x_{n-1}}]\rangle.
\end{align*}
\end{theorem}
\begin{figure}
\centering
$G_n=$
    \begin{tikzpicture}[baseline=6ex, scale=0.8]
    \node[shape=circle,draw=black] (A) at (0,0) {$x_n$};
    \node[shape=circle,draw=black] (B) at (2,0) {$c_n$};
    \node[shape=circle,draw=black] (C) at (4,0) {$y_n$};
    \node[shape=circle,draw=black] (D) at (1,1.73) {$a_n$};
    \node[shape=circle,draw=black] (E) at (3,1.73) {$b_n$};
    \node[shape=circle,draw=black] (F) at (2,1.73*2) {$z_n$} ;
    
    \path [-] (A) edge node[left] {} (B);
    \path [-] (A) edge node[left] {} (D);
    \path [-] (B) edge node[left] {} (D);
    \path [-] (B) edge node[left] {} (C);
    \path [-] (E) edge node[left] {} (B);
    \path [-] (E) edge node[left] {} (C);
    \path [-] (F) edge node[left] {} (E);
    \path [-] (F) edge node[left] {} (D);
    \path [-] (E) edge node[left] {} (D);
    
    \node[] at (1,1.73/3) {$G_{n-1}$};
    \node[] at (3,1.73/3) {$G_{n-1}$};
    \node[] at (2,1.73/3+1.73) {$G_{n-1}$};
\end{tikzpicture}
\caption{The corner points and the junction points of $G_n$.}
\label{fig:match-gn}
\end{figure}
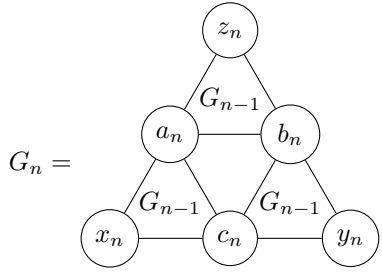
\begin{proof}
First of all, we write the vertex set $V_n$ as the union of 
$$V_n=V_n^{\uparrow}\cup V_n^{\leftarrow}\cup V_n^{\rightarrow}$$ where 
$V_n^{\uparrow}$ (respectively $V_n^{\leftarrow}$ and $V_n^{\rightarrow}$) represents the triangle (as a subset of vertices) of $V_n$  with three corner vertices $\{a_n,b_n,z_n\}$  (respectively $\{x_n,c_n,a_n\}$ and $\{c_n,y_n,b_n\}$), that is, graphically:
\begin{center}
    \begin{tikzpicture}[baseline=6ex, scale=0.8]
    \node[shape=circle,draw=black] (A) at (0,0) {$x_n$};
    \node[shape=circle,draw=black] (B) at (2,0) {$c_n$};
    \node[shape=circle,draw=black] (C) at (4,0) {$y_n$};
    \node[shape=circle,draw=black] (D) at (1,1.73) {$a_n$};
    \node[shape=circle,draw=black] (E) at (3,1.73) {$b_n$};
    \node[shape=circle,draw=black] (F) at (2,1.73*2) {$z_n$} ;
    
    \path [-] (A) edge node[left] {} (B);
    \path [-] (A) edge node[left] {} (D);
    \path [-] (B) edge node[left] {} (D);
    \path [-] (B) edge node[left] {} (C);
    \path [-] (E) edge node[left] {} (B);
    \path [-] (E) edge node[left] {} (C);
    \path [-] (F) edge node[left] {} (E);
    \path [-] (F) edge node[left] {} (D);
    \path [-] (E) edge node[left] {} (D);
    
    \node[] at (1,1.73/3) {$V_n^{\leftarrow}$};
    \node[] at (3,1.73/3) {$V_n^{\rightarrow}$};
    \node[] at (2,1.73/3+1.73) {$V_n^{\uparrow}$};
\end{tikzpicture}
\end{center}
Notice that $V_n^{\uparrow},V_n^{\leftarrow},V_n^{\rightarrow}$ are not mutually disjoint subsets of vertices and their pairwise intersection points are the cut points (junction points).
We define three mappings  $p^{\uparrow},p^{\leftarrow},p^{\rightarrow}$ as following. For every $n\in\N$ let 
\begin{align*}
    & p^{\uparrow}:\Gamma_n\rightarrow\Gamma_{n-1}\big\slash\langle[\delta_{x_{n-1}}],[\delta_{y_{n-1}}]\rangle\\
    & p^{\uparrow}([\eta])=[\eta|_{V^{\uparrow}}].
\end{align*}
We check first that $p^{\uparrow}$ is well-defined. Let $\eta_1,\eta_2\in\Z^{V_n}$ be integer valued vectors indexed over $V_n$ such that there exists $\mathsf{x}\in\Z^{V_n}$ with $\eta_1=\eta_2+\triangle_n \mathsf{x}$.
We have to show now that  $[\eta_1|_{V^{\uparrow}}]=[\eta_2|_{V^{\uparrow}}].$ If we restrict the vector $\triangle_n \mathsf{x}$ indexed over $V_n$ to the vertices of $V^{\uparrow}$ and denote its restriction by $( \triangle_n \mathsf{x})|_{V^{\uparrow}}$, then we have
\begin{align*}
 (\triangle_n\mathsf{x})|_{V^{\uparrow}}=(\triangle_{n-1}\mathsf{x})|_{V^{\uparrow}}-\Big(\sum_{y\sim a_n,y\notin V^{\uparrow}}\mathsf{x}_y\Big)\delta_{\mathsf{x}_{n-1}}-\Big(\sum_{y\sim b_n,y\notin V^{\uparrow}}\mathsf{x}_y\Big)\delta_{y_{n-1}},
\end{align*}
where $\mathsf{x}_y$ represents the entry of $\mathsf{x}$ corresponding to (indexed after) vertex $y$. Then
\begin{align*}
    p^{\uparrow}([\eta_1])&=[\eta_1|_{V^{\uparrow}}]=\Bigg[\eta_2|_{V^{\uparrow}}+\triangle_{n-1}\mathsf{x}|_{V^{\uparrow}}-\Big(\sum_{y\sim a_n,y\notin V^{\uparrow}}\mathsf{x}_y\Big)\delta_{x_{n-1}}-\Big(\sum_{y\sim b_n,y\notin V^{\uparrow}}\mathsf{x}_y\Big)\delta_{y_{n-1}}\Bigg]\\
   & =[\eta_2|_{V^{\uparrow}}]=p^{\uparrow}([\eta_2])
\end{align*}
which shows well-definiteness of $p^{\uparrow}$. It is also easy to see that $p^{\uparrow}$ is a homomorphism, since for 
$\eta_1,\eta_2\in\Z^{V_n}$ it holds
\begin{align*}
    p^{\uparrow}([\eta_1+\eta_2])=[(\eta_1+\eta_2)|_{V^{\uparrow}}]=[\eta_1|_{V^{\uparrow}}+\eta_2|_{V^{\uparrow}}]=[\eta_1|_{V^{\uparrow}}]+[\eta_2|_{V^{\uparrow}}]=p^{\uparrow}([\eta_1])+p^{\uparrow}([\eta_2]).
\end{align*}
In the similar way we define the other two mappings $p^{\leftarrow}, p^{\rightarrow}$ by considering the restrictions of $G_n$ on the left and on the right copy of $G_{n-1}$ respectively, that is, for every $n\in\N$, let 
\begin{align*}
    & p^{\leftarrow}:\Gamma_n\rightarrow\Gamma_{n-1}\big\slash\langle[\delta_{y_{n-1}}],[\delta_{z_{n-1}}]\rangle\quad \text{defined as}\quad  p^{\leftarrow}([\eta])=[\eta|_{V^{\leftarrow}}]\\
    & p^{\rightarrow}:\Gamma_n\rightarrow\Gamma_{n-1}\big\slash\langle[\delta_{z_{n-1}}],[\delta_{x_{n-1}}]\rangle\quad \text{defined as}\quad  p^{\rightarrow}([\eta])=[\eta|_{V^{\rightarrow}}]\\
\end{align*}
which are both well-defined and homomorphisms. Let now 
$$p:\Gamma_n\rightarrow \Gamma^{\uparrow}_{n-1}\oplus \Gamma^{\leftarrow}_{n-1} \oplus \Gamma^{\rightarrow}_{n-1}\quad \text{defined as }\quad p=(p^{\uparrow},p^{\leftarrow},p^{\rightarrow}),$$
which is again a homomorphism. 
The mapping $p$ is also surjective, since for $\mathsf{x},\mathsf{y},\mathsf{z}\in\Z^{V_{n-1}}$, at the two corner points we mod out in the image set. More precisely, take $\mathsf{x}$ to be $0$ in the right and upper corner, $\mathsf{y}$ to be $0$ in the left and upper corner, and $\mathsf{z}$ to be $0$ in the lower two corners. Then we can find a suitable preimage by setting it $0$ at the three cut points and giving it the same values as $\mathsf{x}$ in the lower left copy, as $\mathsf{y}$ in the lower right copy and as $\mathsf{z}$ in the upper copy of $G_n$.

Finding the kernel of $p$ and using the isomorphism theorem for groups implies that the image of $p$ is isomorphic to the quotient group $\Gamma_n\slash \mathsf{Ker}(p)$. In order to get the claim, we have therefore to show that $\mathsf{Ker}(p)=\Big\langle[\mathsf{u}^{\uparrow}],[\mathsf{u}^{\leftarrow}],[\mathsf{u}^{\rightarrow}],[\delta_{a_n}],[\delta_{b_n}],[\delta_{c_n}]\Big\rangle$.
Let $\eta\in\Gamma_n$ with $p([\eta])=([\mathsf{0}],[\mathsf{0}],[\mathsf{0}])$. Then, there exist integers $d_1,d_2,e_1,e_2,f_1,f_2$ and vectors $\mathsf{x},\mathsf{y},\mathsf{z}\in\Z^{V_{n-1}}$ such that
\begin{align*}
    \eta|_{V^{\uparrow}}&=\triangle_{n-1}\mathsf{z}+d_1\delta_{x_{n-1}}+d_2\delta_{y_{n-1}},\\
     \eta|_{V^{\leftarrow}}&=\triangle_{n-1}\mathsf{x}+f_1\delta_{z_{n-1}}+f_2\delta_{y_{n-1}},\\
    \eta|_{V^{\rightarrow}}&=\triangle_{n-1}\mathsf{y}+e_1\delta_{x_{n-1}}+e_2\delta_{z_{n-1}}.
\end{align*}
Denote by $\mathsf{x}'$ the vector in $\Z^{V_n}$ that equals $\mathsf{x}$ on $V^{\leftarrow}$ and $0$ everywhere else. Similarly denote by $\mathsf{y}'$ the vector equal to $\mathsf{y}$ on $V^{\rightarrow}$ and $0$ elsewhere. Finally $\mathsf{z}'$ equals $\mathsf{z}$ on the vertices in $V^{\uparrow}$ and $0$ elsewhere. Let $\mathsf{A}=\mathsf{x}'+\mathsf{y}'+\mathsf{z}'$. Then for any vertex $v\in V_n$ that is neither a cut point nor a neighbour of a cut point we have $\eta(v)=\triangle_n \mathsf{A}(v)$. Let now $u,v\in V^{\uparrow}$ with $u\sim a_n$ and $v\sim a_n$. Then
\begin{align*}
    \triangle_n \mathsf{A}(u)=\triangle_{n-1}\mathsf{z}(u)-\mathsf{x}'(a_n)=\eta(u)-\mathsf{x}'(a_n)\\
    \triangle_n \mathsf{A}(v)=\triangle_{n-1}\mathsf{z}(v)-\mathsf{x}'(a_n)=\eta(v)-\mathsf{x}'(a_n),
\end{align*}
which shows that $\eta$ and $\Delta_n\mathsf{A}$ differ at $u$ and $v$ by the same amount, which is $\mathsf{x}(a_n)$.
Moreover
\begin{align*}
    \triangle_n\delta_{a_n}-4\delta_{a_n}-\mathsf{u}^{\uparrow}=\delta_{u'}+\delta_{v'},
\end{align*}
where $u',v'\in V^{\leftarrow}$, and $u'\sim a_n, v'\sim a_n $, hence we obtain
\begin{align*}
    \delta_{u'}+\delta_{v'}\in\langle[\mathsf{u}^{\uparrow}],[\mathsf{u}^{\leftarrow}],[\mathsf{u}^{\rightarrow}],[\delta_{a_n}],[\delta_{b_n}],[\delta_{c_n}]\big\rangle.
\end{align*}
Doing now the same calculations for the other two cut points $b_n,c_n$ we see
\begin{align*}
    \eta-\triangle_n \mathsf{A}\in\big\langle[\mathsf{u}^{\uparrow}],[\mathsf{u}^{\leftarrow}],[\mathsf{u}^{\rightarrow}],[\delta_{a_n}],[\delta_{b_n}],[\delta_{c_n}]\big\rangle.
\end{align*}
Now a simple calculation shows that $\big\langle[\mathsf{u}^{\uparrow}],[\mathsf{u}^{\leftarrow}],[\mathsf{u}^{\rightarrow}],[\delta_{a_n}],[\delta_{b_n}],[\delta_{c_n}]\big\rangle$ is being mapped to $([\mathsf{0}],[\mathsf{0}],[\mathsf{0}])$, thus we have 
\begin{align*}
    \big\langle[\mathsf{u}^{\uparrow}],[\mathsf{u}^{\leftarrow}],[\mathsf{u}^{\rightarrow}],[\delta_{a_n}],[\delta_{b_n}],[\delta_{c_n}]\big\rangle=\mathsf{Ker}(p)
\end{align*}
and this completes the proof.
\end{proof}

\subsection{Mixing time on Sierpi\'nski gasket graphs}

In this part we show that the mixing time of the sandpile Markov chain on Sierpi\'nski gasket graphs
$G_n=(V_n,E_n)$, $n\in \mathbb{N}$ with normal boundary conditions, is of order $|V_n|\log|V_n|$. For every $n\in\N$, we write $\mathsf{P}_n$ for the transition matrix of the sandpile chain over the sandpile group $\mathcal{R}_n=\mathcal{R}(G_n)$, $\pi_n$ for the uniform distribution on $\mathcal{R}_n$ and $\lambda^{\star}_n$ (respectively $\gamma^{\star}_n$ and $t_{\mathsf{rel}}^n$) for subdominant eigenvalue (respectively spectral gap and relaxation time) of $\mathsf{P}_n$. In order to simplify notation, we will also write $\mathsf{P}^t_{\mathsf{id}_n}:=\mathsf{P}_n^t\delta_{\mathsf{id}_n}$ for the distribution of the chain at time $t$ when it starts at the identity $\mathsf{id}_n$ of $\mathcal{R}_n$. In \cite[Section 2.3]{abel-sand-mix-pike-levine-jerison} the graphs $G_n$ were also considered, and the authors calculated the order of the relaxation time by employing a technique called {\it gadgets} and constructing a suitable multiplicative harmonic function. Together with \cite[Proposition 2.10]{abel-sand-mix-pike-levine-jerison}, the following bound for the mixing time can be obtained:
$$c|V_n|\leq t_{\mathsf{mix}}\leq C |V_n|\log|\mathcal{R}_n|$$
for constants $c,C>0$. Using the number of spanning trees $|\mathcal{R}_n|$ of $G_n$ given by  \eqref{eq:sp-tr-rec}, one gets that the order of the mixing time is between $|V_n|$ and $|V_n|^2$. We improve their bound, by showing that the mixing time is of order $|V_n|\log|V_n|$. 
While the upper bound follows directly from \cite[Theorem 4.3]{abel-sand-mix-pike-levine-jerison}, for the lower bound we use the approach the authors used in order to show cutoff for the sandpile chain on complete graphs. More precisely, we consider a {\it distinguishing statistic $\chi$} for which the distance between the pushforward measure $\mathsf{P}^t_{\mathsf{id}_n}\circ \chi^{-1}$ and $\pi_n\circ \chi^{-1}$ can be bounded from below. By \cite[Proposition 7.8]{peres-mixing-book}, for any $\chi:\mathcal{R}_n\to\mathbb{R}$:
$$\|\mathsf{P}^t_{\mathsf{id}_n}-\pi_n\|_{\mathsf{TV}}\geq 1-\frac{4}{4+R(t)}\quad \text{whenever}\quad R(t)\leq \dfrac{2\Big(\mathbb{E}_{\mathsf{P}^t_{\mathsf{id}_n}}[\chi]-\mathbb{E}_{\pi_n}[\chi]\Big)^2}{|\mathsf{Var}_{\mathsf{P}^t_{\mathsf{id}_n}}[\chi]-\mathsf{Var}_{\pi_n}[\chi]|}.$$

\begin{theorem}\label{thm:mix-time}
The order of the mixing time for the sandpile Markov chain on $G_n$, with $n\geq 2$, with  normal boundary conditions is given by:
$$t_{\mathsf{mix}}=O(n\cdot 3^n).$$
\end{theorem}
\begin{proof}{\it The upper bound} follows immediately from \cite[Theorem 4.3]{abel-sand-mix-pike-levine-jerison}, since $G_n$ is a regular graph with degree $4$, and so
\begin{equation}\label{eq:up-bnd-mix}
\|\mathsf{P}^t_{\mathsf{id_n}}-\pi_n\|_2\leq \frac{1}{4}\quad \text{ for all }\quad  t\geq \frac{5}{4}(|V_n|+1)\log(34|V_n|)
\end{equation}
which implies 
$$t_{\mathsf{mix}}\leq C |V_n|\log|V_n|=Cn 3^n$$
for some constant $C>0$.

{\it The lower bound.} To obtain a matching lower bound, we use eigenfunctions and an adequate choice of a multiplicative harmonic function. Consider the function $h_1:V_1\to\mathbb{T}$ on the vertices of $G_1$ defined as
\begin{center}
        $h_1=$\begin{tikzpicture}[baseline=7ex,scale=0.6]
    \node[shape=circle,draw=black] (A) at (0,0) {1};
    \node[shape=circle,draw=black] (B) at (2,0) {-1};
    \node[shape=circle,draw=black] (C) at (4,0) {1};
    \node[shape=circle,draw=black] (D) at (1,1.73) {-1};
    \node[shape=circle,draw=black] (E) at (3,1.73) {-1};
    \node[shape=circle,draw=black] (F) at (2,1.73*2) {1} ;
    
    \path [-] (A) edge node[left] {} (B);
    \path [-] (A) edge node[left] {} (D);
    \path [-] (B) edge node[left] {} (D);
    \path [-] (B) edge node[left] {} (C);
    \path [-] (E) edge node[left] {} (B);
    \path [-] (E) edge node[left] {} (C);
    \path [-] (F) edge node[left] {} (E);
    \path [-] (F) edge node[left] {} (D);
    \path [-] (E) edge node[left] {} (D);
    
    \node[] at (1,1.73/3) {};
    \node[] at (3,1.73/3) {};
    \node[] at (2,1.73/3+1.73) {};
\end{tikzpicture}
\end{center}
Obviously $h_1$ is multiplicative harmonic on $G_1$. The graph $G_n$ contains $3^{n-1}$ copies of $G_1$. We can extend the function $h_1$ from $G_1$ to a multiplicative harmonic function on $G_n$ as following: choose one of those $3^{n-1}$ copies of $G_1$ in $G_n$, set it to be $h_1$ on this copy, and extend it constantly equal to $1$ on the rest of $G_n$. Then it is obvious that this also is multiplicative harmonic on $G_n$ and the corresponding eigenvalue is $1-\frac{6}{|V_n|+1}$. Order the $3^{n-1}$ subcopies of $G_1$ in $G_n$, and for each $i\in\{1,\ldots,3^{n-1}\}$ denote by $h_n^i$ the multiplicative harmonic function on $G_n$ which equals $h_1$ on the $i$-th copy of $G_1$, and $1$ elsewhere. Let $\chi_n^i$ be the character of $\Gamma_n\cong \mathcal{H}$ corresponding to $h_n^i$, which is in view of \eqref{eq:char-gamma} given by $\chi_n^i(\eta)=\prod_{v\in V_n}\big(h_n^i(v)\big)^{\eta(v)}$, for $\eta\in \mathcal{R}_n$. Using the characters $\chi_n^i$, we consider the {\it distinguishing statistic} $\chi:\mathcal{R}_n\to\mathbb{R}$ given by
$$\chi(\eta):=\frac{1}{3^{n-1}}\sum_{i=1}^{3^{n-1}}\chi_n^i(\eta),$$
which is also in the eigenspace of the eigenvalue $1-\frac{6}{|V_n|+1}$, so $\langle \pi_n, \chi\rangle =0$.
It remains to investigate the expectation and variance of $\chi$ under the distribution $\mathsf{P}^t_{\mathsf{id}_n}$ of the sandpile chain at time $t$ and under stationarity $\pi_n$, respectively. Since $\chi(\mathsf{id}_n)=1$ and $\mathsf{P}_n\chi=\big(1-\frac{6}{|V_n|+1}\big)\chi$, we first have 
$$\mathbb{E}_{\mathsf{P}^t_{\mathsf{id}_n}}[\chi]=\Big(1-\frac{6}{|V_n|+1}\Big)^t\quad \text{and}\quad \mathbb{E}_{\pi_n}[\chi]=0 \quad \text{and}\quad \mathsf{Var}_{\pi_n}[\chi]=\frac{1}{3^{n-1}}.$$
The values for the expectation  $\mathbb{E}_{\pi_n}[\chi]$ and variance $\mathsf{Var}_{\pi_n}[\chi]=\frac{1}{3^{n-1}}$ under stationarity used the fact that $\pi_n$ is a left eigenfunction of $\mathsf{P}_n$ with eigenvalue $1\neq \lambda_{h_n^i}=1-\frac{6}{|V_n|+1}$, since it is the stationary distribution, so for the upper-bound in the numerator of $R(t)$
we have
$$2\Big(\mathbb{E}_{\mathsf{P}^t_{\mathsf{id}_n}}[\chi]-\mathbb{E}_{\pi_n}[\chi]\Big)^2=2\Big(1-\frac{6}{|V_n|+1}\Big)^{2t}.$$
It remains to compute an upper bound for the value in the denominator. In order to calculate $\text{Var}_{P^t_{\text{id}_n}}[\chi]$, we consider for $i,j\in\{1,...,3^{n-1}\}$, with $i\neq j$ the function $\chi^{i,j}_n=\chi_n^i\cdot\chi_n^j$ which is again a character on $\Gamma_n$. If we denote by  $\lambda_n^{i,j}$ the corresponding eigenvalue, then
\begin{align*}
    \lambda_n^{i,j}&=\frac{1}{|V_n|+1}\sum_{v\in V_n\cup\{s\}}\chi_n^{i,j}(\delta_v)
    =\frac{1}{|V_n|+1}\sum_{v\in V_n\cup\{s\}}h_n^i(v)h_n^j(v)
    =1-\frac{12}{|V_n|+1}.
\end{align*}
Notice that $(\chi_n^i)^2$ is the constant $1$ function. When plugging this into the formula for the variance we obtain
\begin{align*}
    \mathsf{Var}_{\mathsf{P}^t_{\mathsf{id}_n}}[\chi]&=\frac{1}{3^{2n-2}}\mathbb{E}_{\mathsf{P}^t_{\mathsf{id}_n}}\Big[\sum_{i=1}^{3^{n-1}}(\chi^i_n)^2\Big]+\frac{1}{3^{2n-2}}\mathbb{E}_{\mathsf{P}^t_{\mathsf{id}_n}}\Bigg[\sum_{i\neq j}\chi_n^{i,j}\Big]-\Bigg(1-\frac{6}{|V_n|+1}\Bigg)^{2t}\\
    &=\frac{1}{3^{n-1}}+\bigg(1-\frac{1}{3^{n-1} }\bigg)\bigg(1-\frac{12}{|V_n|+1}\bigg)^t-\Bigg(1-\frac{6}{|V_n|+1}\Bigg)^{2t}.
\end{align*}
Since  $\Big(1-\frac{6}{|V_n|+1}\Big)^{2}\geq\Big(1-\frac{12}{|V_n|+1}\Big)$ it follows
\begin{align*}
    \big|\mathsf{Var}_{\mathsf{P}^t_{\mathsf{id}_n}}[\chi]-\mathsf{Var}_{\pi_n}[\chi]\big|\leq \frac{2}{3^{n-1}},
\end{align*}
and thus for the function 
\begin{align*}
    R(t)=3^{n-1}\Big(1-\frac{6}{|V_n|+1}\Big)^{2t}
\end{align*}
it holds
\begin{align*}
    \|\mathsf{P}^t_{\mathsf{id}_n}-\pi_n\|_{\mathsf{TV}}\geq 1-\frac{4}{4+R(t)}.
\end{align*}
In order to lower bound $R(t)$, choose $c>0$ which will be specified later, and fix $t\in\N$ such that $0<t<\frac{|V_n|}{12}\log|V_n|-c|V_n|$. This is possible for large enough $n$ because $|V_n|\log|V_n|$ grows faster than $|V_n|$. Then
\begin{align*}
    t\log\Big(1-\frac{6}{|V_n|+1}\Big)\geq\Big(\frac{|V_n|}{12}\log|V_n|-c|V_n|\Big)\log\Big(1-\frac{6}{|V_n|+1}\Big).
\end{align*}
By considering the function $f(x)=\log(1-x)+x+x^2$  we see that $\log(1-x)\geq-x-x^2$ for $x\in[0,\pi^2/16]$. This can be seen by noting that $f(0)=0$ and $f(\pi^2/16)>0$ as well as by considering the derivative of $f$, which implies that $f$ increases on $[0,1/2]$ and decreases on $[1/2,\pi^2/16]$. We then infer that
\begin{align*}
    \log\Big(1-\frac{6}{|V_n|+1}\Big)\geq -\frac{6}{|V_n|}-\frac{36}{|V_n|^2}
\end{align*}
for $n$ large enough, which yields
\begin{align*}
     t\log\Bigg(1-\frac{6}{|V_n|+1}\Bigg)\geq -\frac{\log|V_n|}{2}+6c-3\frac{\log|V_n|-12c}{|V_n|},
\end{align*}
and thus $    R(t)\geq \alpha e^{12c}$,
where $\alpha\leq 3^{n-1}e^{-6\frac{\log|V_n|}{|V_n|}+72\frac{c}{|V_n|}-\log|V_n|}$. The upper bound for $\alpha$ converges to some number strictly greater than $0$, hence we can find an $\alpha$ that fulfills this inequality for all $n$. Choosing $\alpha=2\cdot10^{-6}$ suffices, and this gives
\begin{align*}
   \|\mathsf{P}^t_{\mathsf{id}_n}-\pi_n\|_{\mathsf{TV}}\geq 1-\frac{4\cdot10^6}{4\cdot10^6+2e^{12c}}
\end{align*}
for $0\leq t\frac{|V_n|}{12}\log|V_n|-c|V_n|$. Solving
$     1-\frac{4\cdot10^6}{4\cdot10^6+2e^{12c}}\geq 1/4$,
we  get for  the choice of $c= \log(10^6)/12$  that 
\begin{align*}
    t_{\mathsf{mix}}\geq \frac{|V_n|}{12}\log|V_n|-c|V_n|
\end{align*}
which completes the lower bound and together with the upper bound and with the definition of $V_n$ proves the claim.
\end{proof}
\bibliography{references}{}
\bibliographystyle{alpha}

\textsc{Robin Kaiser}, Department of Mathematics, University of Innsbruck, Austria.\\
\texttt{Robin.Kaiser@uibk.ac.at}

\textsc{Ecaterina Sava-Huss}, Department of Mathematics, University of Innsbruck, Austria.\\
\texttt{Ecaterina.Sava-Huss@uibk.ac.at}

\textsc{Yuwen Wang}, Department of Mathematics, University of Innsbruck, Austria.\\
\texttt{Yuwen.Wang@uibk.ac.at}
\end{document}